\documentclass[11pt]{amsart}
\usepackage{amsthm,amssymb,mathrsfs,mathtools,empheq}
\usepackage[shortlabels]{enumitem}
\usepackage[T1]{fontenc}

\usepackage{fullpage,color}

\newtheorem{mainthm}{Theorem}
\newtheorem{theorem}{Theorem}[section]
\newtheorem*{theorem*}{Theorem}
\newtheorem{corollary}[theorem]{Corollary}

\newtheorem{lemma}[theorem]{Lemma}
\newtheorem{proposition}[theorem]{Proposition}
\newtheorem*{proposition*}{Proposition}

\newtheorem*{conjecture*}{Conjecture}

\theoremstyle{definition}
\newtheorem{definition}[theorem]{Definition}

\newtheorem{remark}[theorem]{Remark}

\numberwithin{equation}{section}

\def\bC {\mathbb{C}}
\def\bN {\mathbb{N}}

\def\bR {\mathbb{R}}

\def\cV {\mathcal{V}}

\def\cY {\mathcal{Y}}

\def\scrL{\mathscr{L}}

\def\grad {{\nabla}}

\def\la {\langle}
\def\ra {\rangle}

\newcommand{\tx}[1]{\mathrm{#1}}

\newcommand{\wt}[1]{\widetilde{#1}}
\newcommand{\bs}[1]{\boldsymbol{#1}}
\newcommand{\conj}[1]{\overline{#1}}

\renewcommand{\div}{\operatorname{div}}

\newcommand{\spn}{\operatorname{span}}

\newcommand{\codim}{\operatorname{codim}}

\renewcommand{\ker}{\operatorname{ker}}

\newcommand{\Id}{\operatorname{Id}}

\newcommand{\eee}{\mathrm e}

\newcommand{\ud}{\mathrm{\,d}}
\newcommand{\vd}{\mathrm{d}}

\newcommand{\dd}[1]{{\frac{\vd}{\vd{#1}}}}

\title{Lyapunov-type characterisation of exponential dichotomies \\
with applications to the heat and Klein-Gordon equations}
\author{Gong Chen}
\address{Dept. of Mathematics, University of Toronto, 40 St. George St., Toronto, ON M5S 2E4, Canada}
\email{gc@math.toronto.edu}
\author{Jacek Jendrej}
\address{CNRS and Universit\'e Paris 13, LAGA, UMR 7539, 99 av J.-B.~Cl\'ement, 93430 Villetaneuse, France}
\email{jendrej@math.univ-paris13.fr}

\begin{document}

\begin{abstract}
We give a sufficient condition for existence of an exponential dichotomy
for a general linear dynamical system (not necessarily invertible) in a Banach space, in discrete or continuous time.
We provide applications to the backward heat equation with a potential varying in time
and to the heat equation with a finite number of slowly moving potentials.
We also consider the Klein-Gordon equation with a finite number of potentials whose centres move at sub-light speed
with small accelerations.
\end{abstract}

\maketitle
\section{Introduction}
\label{sec:intro}
\subsection{Exponential dichotomies}
\label{ssec:setting}
Consider a linear dynamical system
\begin{equation}
\label{eq:intro-dyn}
v_{n+1} = B_n v_n,\qquad n \geq 0,\ v_n \in \bC^d,\ B_n \in \bC^{d \times d}.
\end{equation}
In the special case where $B_n = B \in \bC^{d\times d}$ is independent of $n$
(autonomous dynamics), the dynamical behavior of the solutions of \eqref{eq:intro-dyn} can be described
using the Jordan normal form of the matrix $B$.
In particular, if $B$ has no eigenvalues $\lambda \in \bC$ such that $|\lambda| \in [a, b]$,
where $0 < a < b$, then the phase space $\bC^d$ decomposes as a direct sum
\begin{equation*}
\label{eq:intro-split}
\bC^d = X_\tx s \oplus X_\tx u,
\end{equation*}
where $X_\tx s$ and $X_\tx u$ are invariant for \eqref{eq:intro-dyn}
and there exist constants $c, C > 0$ such that
\begin{itemize}
\item if $v_0 \in X_\tx s$, then $|v_n| \leq C a^n |v_0|$ for all $n \geq 0$,
\item if $v_0 \in X_\tx u$, then $|v_n| \geq c\, b^n |v_0|$ for all $n \geq 0$.
\end{itemize}
Such a situation is called an \emph{exponential dichotomy}.
We call $X_\tx s$ the \emph{stable subspace} and $X_\tx u$ the \emph{unstable subspace}.

The purpose of this paper is to construct exponential dichotomies for \eqref{eq:intro-dyn}
and similar systems in the case where $B_n$ changes with $n$.
There are many classical examples exhibiting ``surprising'' behavior of the system \eqref{eq:intro-dyn}.
One such example is given by
\begin{equation*}
B_{2m} = \begin{pmatrix} 0 & -2 \\ 1/8 & 0 \end{pmatrix},
\qquad B_{2m+1} = \begin{pmatrix} 0 & -1/8 \\ 2 & 0 \end{pmatrix}.
\end{equation*}
It is easy to see that the eigenvalues of $B_{2m}$ and $B_{2m+1}$ are $\pm \frac i2$,
and thus have modulus $< 1$.
However, the eigenvalues of $B_1 B_0$ are $-\frac{1}{64}$ and $-4$,
and it turns out that if $x_2 \neq 0$, then the initial data $v_0 = (x_1, x_2)$
yields to exponential growth of the sequence $(v_n)$.
This example shows that the spectra of $B_n$ do not provide enough information to describe
exponential dichotomies of \eqref{eq:intro-dyn}.
Indeed, it is necessary to control how contracting/expanding directions relate to each other
as $n$ changes.

There exists an extensive literature on exponential dichotomies for non-autonomous dynamical systems.
The monograph by Coppel \cite{Coppel} deals with the case of linear ordinary differential equations.
In particular, it provides a necessary and sufficient condition for an exponential dichotomy
in terms of existence of a Lypaunov functional satisfying certain properties.
Related results were obtained by Coppel~\cite{Coppel-2}, Muldowney \cite{Muldowney} and,
in the case of difference equations, Papaschinopoulos \cite{Papaschinopoulos}.

A different approach to exponential dichotomies is based on the \emph{evolution semigroup} introduced by Howland~\cite{Howland},
which means that the non-autonomous system is transformed to an operator semigroup on some space,
whose properties are then studied using spectral methods.
This theory, both in finite and infinite dimension,
is developed in the works of Rau~\cite{Rau}, Latushkin and Montgomery-Smith~\cite{LM},
R\"abiger and Schnaubelt~\cite{RS2}, as well as subsequent works.
One can consult the monograph~\cite{CL-book} for a comprehensive bibliography.

The works \cite{Coppel, Coppel-2, Muldowney, Papaschinopoulos} mentioned above do not seem to directly generalise to infinite dimension.
However, a Lyapunov-type characterisation of exponential dichotomies in infinite-dimensional
Hilbert spaces was obtained by Barreira, Dragi{\v c}evi\'c  and Valls~\cite{BDV17},
using the theory of evolution semigroups, see also~\cite{BDV-book}.

In this paper, we adopt the Lyapunov-type approach and formulate conditions for existence of exponential dichotomies
in terms of existence of \emph{Lyapunov} (or \emph{energy}) \emph{functionals} satisfying certain properties.
Instead of invoking the evolution semigroup theory, we provide an alternative and more direct method,
which we believe can be useful in applications.
One advantage of our method is that in many cases we can easily obtain some supplementary information about the (un)stable spaces,
for example the (co)dimension or an approximate basis.

In the theory of linear cocycles, exponential dichotomies are related
to the existence of the so-called \emph{Oseledets flag}, see \cite{Viana}.
Our proof of existence of exponential dichotomies
resembles known proofs of the Oseledets Theorem, especially the one given in \cite{DuarteKlein}.

Finally, would like to point out that one of the important properties of exponential dichotomies
is that they often persist under (not necessarily linear) perturbations of the dynamical system.
This general principle is called the \emph{Lyapunov-Perron method}, see for instance \cite{BarreiraValls}.

\subsection{Statement of the results}
\label{ssec:statement}
Because we are interested in applications to dynamics of partial differential equations,
we need to work with an infinite dimensional phase space.
As we are not going to rely on Spectral Theory, we take it to be a real Banach space denoted $X$.
Let $B_n \in \scrL(X)$, $n \in \{0, 1, \ldots\}$ be a sequence of bounded linear operators on $X$.
We consider the dynamical system
\begin{equation}
\label{eq:intro-dynsyst-2}
v_{n+1} = B_n v_n,\qquad v_0 \in X.
\end{equation}
For $n \leq m$ we denote
\begin{equation*}
B(n, n) := \tx{Id}, \quad B(m, n) := B_{m-1} B_{m-2}\ldots  B_{n}.
\end{equation*}
Note that we do not require boundedness of the sequence $(B_n)$ in $\scrL(X)$.
\begin{definition}
\label{def:intro-exp-dich-2}
We say that \eqref{eq:intro-dynsyst-2} has an \emph{exponential dichotomy} with values $a$ and $b$,
$0 < a < b$, if for all $n \geq n_0$ there exists a direct sum decomposition $X = X_\tx s(n) \oplus X_\tx u(n)$
such that $X_\tx s(n)$, $X_\tx u(n)$ and the associated projections
$\pi_\tx s(n): X \to X_\tx s(n)$ and $\pi_\tx u(n): X \to X_\tx u(n)$ have the following properties
for some $C > 0$ and all $n \leq m$:
\begin{enumerate}
\item $B(m, n)\circ\pi_\tx s(n) = \pi_\tx s(m)\circ B(m, n)$ and $B(m, n)\circ\pi_\tx u(n) = \pi_\tx u(m)\circ B(m, n)$,
\item $\|\pi_\tx s(n)\|_{\scrL(X)} + \|\pi_\tx u(n)\|_{\scrL(X)} \leq C$,
\item $B(m, n)\vert_{X_\tx u(m)}: X_\tx u(n) \to X_\tx u(m)$ is invertible,
\item $\|B(m, n)v_n\| \leq C a^{m-n}\|v_n\|$ for all $v_n \in X_\tx s(n)$,
\item $\|B(m, n)^{-1}v_m\| \leq C b^{n-m}\|v_m\|$ for all $v_m \in X_\tx u(m)$.
\end{enumerate}
\end{definition}
\begin{remark}
It is clear that $X_\tx s(n)$ is unique. In general, $X_\tx u(n)$ is not unique.
\end{remark}
Our sufficient condition for existence of an exponential dichotomy is expressed in terms
of two sequences of (nonlinear) continuous homogeneous functionals $I_n^-, I_n^+ : X \to \bR_+$.
Given $I_n^-$, $I_n^+$ and a number $c > 0$, we define the stable and the unstable cone
\begin{gather}
\cV_\tx s(c, n) := \{v \in X: I_n^+(v) \leq c I_n^-(v)\}, \label{eq:Vs-def} \\
\cV_\tx u(c, n) := \{v \in X: I_n^+(v) \geq c I_n^-(v)\}. \label{eq:Vu-def}
\end{gather}
We find it helpful to keep in mind that if $c$ is small, then $\cV_\tx s(c, n)$ is ``thin'' and $\cV_\tx u(c, n)$ is ``wide''.
Conversely, if $c$ is large, then $\cV_\tx s(c, n)$ is ``wide'' and $\cV_\tx u(c, n)$ is ``thin''.

Firstly, we assume that there exists $c_1 > 0$ (independent of $n$) such that
\begin{equation}
\label{eq:norm-equiv}
c_1 \|v\|_X \leq I_n^-(v) + I_n^+(v) \leq \frac{1}{c_1}\|v\|_X,
\qquad \text{for all }n \geq 0 \text{ and }v \in X.
\end{equation}
Note that, directly from the definitions above, we obtain
\begin{align}
v \in \cV_\tx s(c, n)\ &\Rightarrow\ c_1 I_n^-(v) \leq \|v\|_X \leq \frac{1+c}{c_1}I_n^-, \label{eq:norm-equiv-Im} \\
v \in \cV_\tx u(c, n)\ &\Rightarrow\ c_1 I_n^+(v) \leq \|v\|_X \leq \frac{1+c}{c_1 c}I_n^+, \label{eq:norm-equiv-Ip}
\end{align}
thus on the stable cone the norm is equivalent to $I_n^-$, and on the unstable cone it is equivalent to $I_n^+$.

Secondly, we assume that there exist $c_2 > 0$, $K \in \{0, 1, 2, \ldots\}$
and $\alpha_{k, n}^+ \in X^*$ for $(k, n) \in \{1, \ldots, K\}\times \{0, 1, \ldots\}$ such that
\begin{equation}
\label{eq:Inp-form}
c_2 \max_{1 \leq k \leq K}|\la \alpha_{k, n}^+, v\ra| \leq I_n^+(v) \leq \frac{1}{c_2}\max_{1 \leq k \leq K}|\la \alpha_{k, n}^+, v\ra|.
\end{equation}

Lastly, 
we assume that there exist $c_3, c_4 > 0$ and $0 < a < b < \infty$ such that
\begin{align}
&\cV_\tx u(c_3, n)\text{ contains a linear space of dimension }K\text{ for all }n, \label{eq:contains-K-2} \\
&c_4 < \frac 13(c_1c_2)^2\frac{c_3}{1+c_3}, \label{eq:intro-c3c4-2} \\
&I_{n+1}^-(B_n v_n) \leq a I_{n}^-(v_n)\quad\text{if }B_n v_n \in \cV_\tx s(c_3, n+1), \label{eq:intro-ineq-Im-20} \\
&I_{n+1}^+(B_n v_n) \geq b I_{n}^+(v_n)\quad\text{if }v_n \in \cV_\tx u(c_4, n). \label{eq:intro-ineq-Ip-20}
\end{align}

\begin{mainthm}
\label{thm:main-dis-2}
Under assumptions \eqref{eq:norm-equiv}, \eqref{eq:Inp-form}--\eqref{eq:intro-ineq-Ip-20},
the system \eqref{eq:intro-dynsyst-2} has an exponential dichotomy with values $a$ and $b$.
For all $n \geq 0$ the stable subspace $X_\tx s(n)$ is contained in $\cV_\tx s(c_4, n)$ and has codimension $K$.
\end{mainthm}
\begin{remark}
Note that if \eqref{eq:intro-ineq-Im-20} holds, then it also holds with $c_3$ replaced by any smaller number.
Similarly, if \eqref{eq:intro-ineq-Ip-20} holds, then it also holds with $c_4$ replaced by any bigger number.
In other words, \eqref{eq:intro-ineq-Im-20} and \eqref{eq:intro-ineq-Ip-20} imply
\begin{align}
I_{n+1}^-(B_n v_n) \leq a I_{n}^-(v_n)\quad&\text{if }B_n v_n \in \cV_\tx s(c, n+1), \label{eq:intro-ineq-Im-2} \\
I_{n+1}^+(B_n v_n) \geq b I_{n}^+(v_n)\quad&\text{if }v_n \in \cV_\tx u(c, n) \label{eq:intro-ineq-Ip-2}
\end{align}
for all $c \in [c_4, c_3]$.
\end{remark}
\begin{remark}
One can show that \eqref{eq:contains-K-2} always holds if
$\alpha_{k, n}^+$ are uniformly linearly independent and $c_3$ small enough, see Proposition~\ref{prop:Kspace-2}.

The condition \eqref{eq:intro-c3c4-2} that we impose on $c_4$ is far from being optimal.
In the applications, it only matters that $c_4$ is required to be smaller than $c_3$
multiplied by some small positive constant depending on $c_1, c_2, K$.
\end{remark}
\begin{remark}
 The appropriate energy functionals $I_n^-$ and $I_n^+$ are constructed
in each particular case using the specific structure of a given problem, and in particular
the natural energy functionals associated with it.
Intuitively, we would like $I_n^-$ to control the ``shrinking'' in the stable direction.
Similarly, $I_n^+$ has to control the ``expanding'' in the unstable directions.
Note that the assumption of the stable/unstable component being significant
is ``before the step'' in the direction of expansion.
The condition \eqref{eq:Inp-form} means that there are only finitely many expansion directions,
which is true for any application we could think of.
If $c_4$ is small, then $X_\tx s(n) \subset \cV_\tx s(c_4, n)$ gives a precise information about
the subspace $X_\tx s(n)$.
\end{remark}

To complete our analysis, we will prove that existence of an exponential dichotomy implies existence of
energy functionals $I_n^-$ and $I_n^+$ satisfying conditions which are apparently stronger than
the conditions listed above (thus, in reality, equivalent).
\begin{proposition}
\label{prop:cond-nec}
If \eqref{eq:intro-dynsyst-2} has an exponential dichotomy with values $a < b$, then there exist semi-norms $\wt I_n^-$, $\wt I_n^+$
satisfying \eqref{eq:norm-equiv} and
\begin{gather*}
X_\tx s(n) = \{v_n: \wt I_n^+(v_n) = 0\}, \\
X_\tx u(n) = \{v_n: \wt I_n^-(v_n) = 0\}, \\
\wt I_n^-(B_n v_n) \leq a \wt I_n^-(v_n)\quad\text{for all }v_n \in X, \\
\wt I_n^+(B_n v_n) \geq b \wt I_n^+(v_n)\quad\text{for all }v_n \in X.
\end{gather*}
If $X_\tx s(n)$ has finite codimension $K$, then $\wt I_n^+$ satisfies \eqref{eq:Inp-form}.
\end{proposition}

We will also consider the case of a \emph{backward} dynamical system
\begin{equation}
\label{eq:dynsyst}
v_{n-1} = A_n v_n, \qquad v_0 \in X.
\end{equation}
We do not require $A_n$ to be invertible or the sequence $(A_n)$ to be bounded in $\scrL(X)$.
For $n \leq m$ we denote
\begin{equation*}
A(n, n) := \tx{Id}, \quad A(n, m) := A_{n+1} A_{n+2}\ldots A_{m}.
\end{equation*}
\begin{definition}
\label{def:exp-dich}
We say that \eqref{eq:intro-dynsyst-2} has a (uniform) \emph{exponential dichotomy} with values $a$ and $b$,
$0 < a < b$, if for all $n \geq n_0$ there exists a direct sum decomposition $X = X_\tx s(n) \oplus X_\tx u(n)$
such that $X_\tx s(n)$, $X_\tx u(n)$ and the associated projections
$\pi_\tx s(n): X \to X_\tx s(n)$ and $\pi_\tx u(n): X \to X_\tx u(n)$ have the following properties for all $n \leq m$:
\begin{enumerate}
\item $A(n, m)\circ\pi_\tx s(m) = \pi_\tx s(n)\circ A(n, m)$ and $A(n, m)\circ\pi_\tx u(m) = \pi_\tx u(n)\circ A(n, m)$,
\item there exists a constant $C$ such that $\|\pi_\tx s(n)\|_{\scrL(X)} + \|\pi_\tx u(n)\|_{\scrL(X)} \leq C$,
\item $A(n, m)\vert_{X_\tx s(m)}: X_\tx s(m) \to X_\tx s(n)$ is invertible,
\item there exists a constant $C$ such that $\|A(n, m)^{-1}v_n\| \leq C a^{m-n}\|v_n\|$ for all $v_n \in X_\tx s(n)$,
\item there exists a constant $C$ such that $\|A(n, m)v_m\| \leq C b^{n-m}\|v_m\|$ for all $v_m \in X_\tx u(m)$.
\end{enumerate}
\end{definition}

We make the following assumptions about the functionals $I_n^\pm$.
We assume that there exist $c_2 > 0$, $K \in \{0, 1, 2, \ldots\}$
and $\alpha_{k, n}^- \in X^*$ for $(k, n) \in \{1, \ldots, K\}\times \{0, 1, \ldots\}$ such that
\begin{equation}
\label{eq:Inm-form}
c_2 \max_{1 \leq k \leq K}|\la \alpha_{k, n}^-, v\ra| \leq I_n^-(v) \leq \frac{1}{c_2}\max_{1 \leq k \leq K}|\la \alpha_{k, n}^-, v\ra|.
\end{equation}

We define the stable and unstable cone by the same formulas \eqref{eq:Vs-def} and \eqref{eq:Vu-def}.
Instead of \eqref{eq:contains-K-2}--\eqref{eq:intro-ineq-Ip-20}, we assume
\begin{align}
&\cV_\tx s(c_3, n)\text{ contains a linear space of dimension }K\text{ for all }n, \label{eq:contains-K} \\
&c_4 > 3(c_1c_2)^{-2}(c_3+1), \label{eq:intro-c3c4} \\
&I_{n}^-(v_n) \leq a I_{n-1}^-(A_n v_n)\quad\text{if }v_n \in \cV_\tx s(c_4, n), \label{eq:intro-ineq-Im0} \\
&I_{n}^+(v_n) \geq b I_{n-1}^+(A_n v_n)\quad\text{if }A_n v_n \in \cV_\tx u(c_3, n-1). \label{eq:intro-ineq-Ip0}
\end{align}
\begin{mainthm}
\label{thm:main-dis}
Under assumptions \eqref{eq:norm-equiv}, \eqref{eq:Inm-form}--\eqref{eq:intro-ineq-Ip0},
the system \eqref{eq:dynsyst} has an exponential dichotomy with values $a$ and $b$.
For all $n \geq 0$ the stable subspace $X_\tx s(n)$ is contained in $\cV_\tx s(c_3, n)$ and has dimension $K$.
\end{mainthm}
\begin{remark}
As before, \eqref{eq:intro-ineq-Im0} and \eqref{eq:intro-ineq-Ip0} imply
\begin{align}
I_{n}^-(v_n) \leq a I_{n-1}^-(A_n v_n)\quad&\text{if }v_n \in \cV_\tx s(c, n), \label{eq:intro-ineq-Im} \\
I_{n}^+(v_n) \geq b I_{n-1}^+(A_n v_n)\quad&\text{if }A_n v_n \in \cV_\tx u(c, n-1). \label{eq:intro-ineq-Ip}
\end{align}
for all $c \in [c_3, c_4]$.
\end{remark}
\begin{remark}
Note that if $c_3, c_4$ are large, then $\cV_\tx s(c_4, n)$ is a wide cone and $\cV_\tx u(c_3, n-1)$ is a thin cone.
We will prove in Proposition~\ref{prop:Kspace} that \eqref{eq:contains-K} holds if $\alpha_{k, n}^-$ are uniformly linearly independent
and $c_3$ is large enough.
However, it is often possible to use a much smaller value of $c_3$, which gives more information about $X_\tx s(n)$.
Again, \eqref{eq:intro-c3c4} is not optimal.
\end{remark}
One can state and prove analogous results for continuous dynamical systems, see Sections~\ref{ssec:projections-cont}
and \ref{ssec:projections-cont-2}.

\subsection{Applications in PDE}
\label{ssec:applications}
As a typical application, we can think of a heat equation with a time-dependent
potential in the case of a forward dynamical system and of a backward heat equation with a time-dependent potential
in the case of a backward dynamical system.
We explain in Section~\ref{sec:examples} how to apply our result to the heat equation in the following two situations:
\begin{itemize}
\item the potential is almost constant on short time intervals in a suitable $L^p$ norm,
\item a potential of a fixed shape (or a finite number of such potentials) are moving in space with a small velocity.
\end{itemize}
Our result in the first case is quite similar to previous results of Schnaubelt~\cite{Schnaubelt-1, Schnaubelt-2}.

In Section~\ref{sec:wave}, we apply the general results to the Klein-Gordon equation with moving potentials.
To our knowledge, this is a first non-trivial example of exponential dichotomies for a wave-type equation.

Our main motivation is the study of multi-solitons for nonlinear models. In this situation,
the potential is given by linearising the equation around an approximate solution.
The hyperbolic structure of the flow around this approximate solution can often by obtained
by the Lyapunov-Perron method if the existence of exponential dichotomy for the linear model is proved.
We believe that this approach could lead to an alternative construction of multi-solitons
in the weak interaction regime, see for instance \cite{Martel05, Combet, CoMu2014, CoMa-multi, MRT-water}.

\subsection{Acknowledgments}
We would like to thank Wilhelm Schlag for helpful discussions.
Part of this work was completed when the second author was visiting the University of Chicago Mathematics Department
and the University of Toronto Mathematics Department.
We would like to thank the Beijing International Center for Mathematical Research, where this work was finished.


\section{Constructions of exponential dichotomies}
\label{sec:general}
\subsection{Discrete backward dynamical systems}
\label{ssec:projections-dis}
In this section we prove Theorem~\ref{thm:main-dis}.
We use the so-called \emph{method of invariant cones}, cf. \cite[Section 4.4.2]{Viana}.
\begin{lemma}
\label{lem:V-inv}
For all $n \geq 1$ and $c \in [c_3, c_4]$ there is
\begin{gather*}
A_n^{-1} \cV_\tx u(c, n-1)\subset \cV_\tx u(c, n), \\
\conj{A_n\cV_\tx s(c, n)} \subset \cV_\tx s(c, n-1).
\end{gather*}
\end{lemma}
\begin{proof}
In order to prove the first inclusion, suppose $v_n \in X$ is such that $A_n v_n \in \cV_\tx u(c, n-1)$
and $v_n \notin \cV_\tx u(c, n)$, thus $v_n \in \cV_\tx s(c, n)$.
From \eqref{eq:intro-ineq-Im} and \eqref{eq:intro-ineq-Ip} we get
\begin{equation*}
c I_{n-1}^-(A_n v_n) \geq \frac ca I_n^-(v_n) > \frac{1}{a}I_n^+(v_n) \geq \frac{b}{a}I_{n-1}^+(A_n v_n)
\geq I_{n-1}^+(A_n v_n),
\end{equation*}
which contradicts $A_n v_n \in \cV_\tx u(c, n-1)$.

In order to prove the second inclusion, suppose $v_n \in X$ is such that $v_n \in \cV_\tx s(c, n)$
and $A_n v_n \notin \cV_\tx s(c, n-1)$, thus $A_n v_n \in \cV_\tx u(c, n-1)$.
From \eqref{eq:intro-ineq-Im} and \eqref{eq:intro-ineq-Ip} we get
\begin{equation*}
c I_{n-1}^-(A_n v_n) \geq \frac{c}a I_n^-(v_n) \geq \frac{1}{a}I_n^+(v_n) \geq \frac{b}{a}I_{n-1}^+(A_n v_n)
\geq I_{n-1}^+(A_n v_n),
\end{equation*}
which contradicts $A_n v_n \notin \cV_\tx s(c, n-1)$.
\end{proof}
For $c \in [c_3, c_4]$ we define the stable subspace by
\begin{equation}
\label{eq:Xs-def}
X_\tx s(c, n) := \bigcap_{n' > n}\conj{A(n, n')\cV_\tx s(c, n')}.
\end{equation}
Clearly, $X_\tx s(c, n)$ is a closed set, but it is not even obvious if $X_\tx s(c, n)$ is a linear subspace of $X$.
\begin{lemma}
\label{lem:Xs-charact}
For all $n \geq 0$ and $c \in [c_3, c_4]$ the following conditions are equivalent:
\begin{enumerate}[(i)]
\item
\label{it:Xs-charact-i}
$w \in X_\tx s(c, n)$,
\item
\label{it:Xs-charact-ii}
there exists $C \geq 0$ such that for all $\delta \in (0, 1)$ and $n' \geq n$
there is a solution $(v_n, \ldots, v_{n'})$ of \eqref{eq:dynsyst} satisfying $\|v_n - w\|_X \leq \delta$
and $\|v_{n'}\|_X \leq C a^{n'-n}$,
\item
\label{it:Xs-charact-iii}
there exist $C \geq 0$ and $d < b$ such that for all $\delta \in (0, 1)$ and $n' \geq n$
there is a solution $(v_n, \ldots, v_{n'})$ of \eqref{eq:dynsyst} satisfying $\|v_n - w\|_X \leq \delta$
and $\|v_{n'}\|_X \leq C d^{n'-n}$.
\end{enumerate}
\end{lemma}
\begin{proof}
\ref{it:Xs-charact-i} $\Rightarrow$ \ref{it:Xs-charact-ii}.
Let $w \in X_{\tx s}(c, n)$, $\delta \in (0, 1)$, $n' > n$.
By the definition of $X_{\tx s}(c, n)$,
there exists $v_{n'} \in V_{\tx s}(c, n')$ such that $\|A(n, n')v_{n'} - w\|_{X} \leq \delta$.
We will show that there exists $C \geq 0$, independent of $\delta$ and $n'$, such that
$\|v_{n'}\|_{X}\leq C a^{n'-n}$.

Consider $v_m := A(m, n')v_{n'}$ for $m \in \{n, \ldots, n'-1\}$.
Then, by Lemma~\eqref{lem:V-inv}, $v_m \in V_{\tx s}(c, m)$ for all $m \in \{n, \ldots, n'-1\}$.
Since $v_m \in V_{\tx s}(c, m)$ for $n \leq m < n'$, \eqref{eq:intro-ineq-Im}
implies $I_m^-(v_m) \leq a I_{m-1}^-(A_m v_m)$ for $n < m < n'$,
which yields $I_{n'}^-(v_{n'}) \leq a^{n'-n}I_n^-(v_n)$.
From \eqref{eq:norm-equiv-Im} we obtain
\begin{equation*}
\|v_{n'}\|_X \leq \frac{1+c}{c_1}I_{n'}^-(v_{n'}) \leq \frac{1+c}{c_1}a^{n'-n}I_n^-(v_n) \leq \frac{1+c}{c_1^2}a^{n'-n}\|v_n\|,
\end{equation*}
which proves \ref{it:Xs-charact-ii} with $C = \frac{(1+c)(1+ \|w\|_X)}{c_1^2}$.

\ref{it:Xs-charact-ii} $\Rightarrow$ \ref{it:Xs-charact-iii} follows from $a < b$.

\ref{it:Xs-charact-iii} $\Rightarrow$ \ref{it:Xs-charact-i}.
Let $v_n \in X$ be such that \ref{it:Xs-charact-iii} holds.
Suppose that $w \notin X_\tx s(c, n)$. This means that there exists $m > n$ such that
\begin{equation*}
\label{eq:notinVs}
w \notin \conj{A(n, m) \cV_\tx s(c, m)}.
\end{equation*}
In other words, there exists $\delta > 0$ such that if $(v_n, \ldots, v_m)$ is a solution of \eqref{eq:dynsyst}
such that $\|v_n - w\|\leq \delta$, then $v_m \in X \setminus \cV_\tx s(c, m) \subset \cV_\tx u(c, m)$.
We fix this $\delta$ (without loss of generality assume $\delta < \frac 12 \|w\|$)
and let $(v_n, \ldots, v_{n'})$ be a solution of \eqref{eq:dynsyst}
having the properties described in \ref{it:Xs-charact-iii}, with $n'$ large.
Since $v_m \in \cV_\tx u(c, m)$, Lemma~\ref{lem:V-inv} yields $v_k \in \cV_\tx u(c, k)$
for $k \in \{m, \ldots, n'\}$. Thus
\begin{equation*}
\frac{C}{c_1}d^{n'-n}\geq \frac{1}{c_1}\|v_{n'}\| \geq I_{n'}^+(v_{n'}) \geq b^{n' - m} I_m^+(v_m)
\geq\frac{c_1c}{1+c}b^{n'-m}\|v_m\|.
\end{equation*}
Since we assume $d < b$, by taking $n'$ sufficiently large we can ensure that
\begin{equation*}
\|v_m\| \leq \frac{\delta}{\|A_{n+1}\|\ldots\|A_{m}\|}.
\end{equation*}
This implies $\|v_n\| \leq \delta$, thus $\|w\| \leq \|v_n\| + \delta \leq 2\delta$,
contradicting the choice of $\delta$.
\end{proof}
\begin{remark}
In the proof of the last lemma, assumptions \eqref{eq:Inm-form}--\eqref{eq:intro-c3c4} were not used.
\end{remark}
\begin{corollary}
\label{cor:Xs}
For all $n \geq 0$, the set $X_\tx s(c, n) = X_\tx s(n)$ does not depend on $c \in [c_3, c_4]$.
It is a closed linear subspace of $X$
and $A_n: X_\tx s(n) \to X_\tx s(n-1)$ is a linear embedding.
\end{corollary}
\begin{proof}
Condition~\ref{it:Xs-charact-ii} in Lemma~\ref{lem:Xs-charact} defines a linear subspace
independent of $c$ so, by Lemma~\ref{lem:Xs-charact}, $X_\tx s(n)$ is a linear subspace of $X$.
We see directly from \eqref{eq:Xs-def} that it is closed and that $A_n w \in X_\tx s(n-1)$
whenever $w \in X_\tx s(n)$. The fact that $A_n$ is an embedding on $X_\tx s(n)$ follows
from \eqref{eq:intro-ineq-Im} and the fact that $\|\cdot\|$ is comparable to $I_n^-$ on $\cV_\tx s(c, n)$.
\end{proof}
\begin{proof}[Proof of Theorem~\ref{thm:main-dis}]
We set $X_\tx u(0) := \bigcap_{k=1}^K \ker \alpha_{k, 0}^-$ and we define inductively
\begin{equation*}
X_\tx u(n) := A_n^{-1}(X_\tx u(n-1)), \quad\text{for }n > 0.
\end{equation*}
By the definition of $\cV_\tx u(c_3, n)$ and \eqref{eq:Inm-form}, we have
$\bigcap_{k=1}^K \ker \alpha_{k, 0}^- \subset \cV_\tx u(c_4, 0)$, thus Lemma~\ref{lem:V-inv} yields
$X_{\tx u}(n) \subset \cV_\tx u(c_4, n)$, for all $n\geq 0$.
For all $n \geq 0$, $X_\tx u(n)$ is a linear subspace of $X$ of codimension at most $K$
(as we will see later, in fact equal to $K$).

Note that the choice of $X_\tx u(0)$ is not canonical,
in fact we could take as $X_\tx u(0)$ any subspace of codimension $\leq K$ contained in $\cV_\tx u(c_4, 0)$.

We will find a constant $c_5 > 0$ depending on $c_1, \ldots, c_4$
such that if $v \in X_\tx s(n)$ and $w \in X_\tx u(n)$,
then
\begin{equation}
\label{eq:v+w-lbound}
\|v + w\| \geq c_5 \|v\|.
\end{equation}
If $\|w\|\geq \frac 32 \|v\|$, then \eqref{eq:v+w-lbound} follows from the triangle inequality.
Assume $\|w\| \leq \frac 32 \|v\|$.
Since $X_\tx s(n) \subset \cV_\tx s(c_3, n)$, \eqref{eq:norm-equiv-Im} yields $I_n^-(v) \geq \frac{c_1}{1+c_3}\|v\|_X$,
thus by \eqref{eq:Inm-form} there is $k_0 \in \{1, \ldots, K\}$ such that
\begin{equation}
\label{eq:alphav-big}
|\la \alpha_{k_0}^-, v\ra| \geq \frac{c_1c_2}{1+c_3}\|v\|.
\end{equation}
Since $X_\tx u(n) \subset \cV_\tx u(c_4, n)$, we have $c_4 I_n^-(w) \leq I_n^+(w)$,
so \eqref{eq:norm-equiv} yields $(1+c_4)I_n^-(w) \leq \frac{1}{c_1}\|w\| \leq \frac{3}{2c_1}\|v\|$.
Invoking again \eqref{eq:Inm-form} we obtain
\begin{equation}
\label{eq:alphaw-small}
|\la \alpha_{k_0}^-, w\ra| \leq \frac{3}{2c_1c_2(1+c_4)}\|v\|.
\end{equation}
From \eqref{eq:alphav-big} and \eqref{eq:alphaw-small} we get
\begin{equation*}
\|v + w\| \geq c_1 I_n^-(v+w) \geq c_1c_2|\la \alpha_{k_0}^-, v+w\ra| \geq c_1c_2\Big(\frac{c_1c_2}{1+c_3} - \frac{3}{2c_1c_2(1+c_4)}\Big)\|v\|.
\end{equation*}
Assumption \eqref{eq:intro-c3c4} implies that the constant in front of $\|v\|$ is $> 0$, so we have proved \eqref{eq:v+w-lbound}.

Next, we prove that $X = X_\tx s(n) \oplus X_\tx u(n)$.
Bound \eqref{eq:v+w-lbound} directly yields $X_\tx s(n) \cap X_{\tx u}(n) = \{0\}$.

Let $v \in X$. Let $\Pi := v + X_\tx u(n)$. By assumption \eqref{eq:contains-K},
for any $n' \geq n$ the cone $\cV_\tx s(c_3, n')$ contains a linear space $\wt X_\tx s(n')$ of dimension $K$.
We see that $A(n, n')\vert_{\wt X_\tx s(n')}$ is one-to-one. Indeed, since $\ker A(n, n') \subset \cV_\tx u(c_4, n')$
and $\wt X_\tx s(n') \subset \cV_\tx s(c_3, n')$, this follows from $\cV_\tx s(c_3, n') \cap \cV_\tx u(c_4, n') = \{0\}$.

Since $X_\tx u(n) \cap A(n, n')\wt X_\tx s(n') = \{0\}$
and $\codim(X_\tx u(n)) \leq K$, we actually have
\begin{equation}
\label{eq:Xu-codim}
\codim(X_\tx u(n)) = K
\end{equation}
and the intersection $\Pi \cap A(n, n')\wt X_\tx s(n')$ is non-empty.
In particular, $\Pi \cap \conj{A(n, n')\cV_\tx s(c_3, n')}$ is a nested family of closed non-empty sets.
It suffices to show that their diameters tend to 0 as $n' \to \infty$.

Let $w_1, w_2 \in \Pi \cap \conj{A(n, n')\cV_\tx s(c_3, n')}$. Then $w := w_1 - w_2 \in X_\tx u(n)$.
Let $\delta > 0$. There exist $w_1', w_2' \in \cV_\tx s(c_3, n')$
such that $\|A(n, n')w_k' - w_k\| \leq \delta$ for $k \in \{1, 2\}$.
Using Lemma~\ref{lem:V-inv}, \eqref{eq:intro-ineq-Im} and \eqref{eq:norm-equiv}, we have
\begin{equation*}
I_{n'}^-(w_k') \leq a^{n'-n}I_n^-(A(n, n')w_k')\leq \frac{1}{c_1}a^{n'-n}\|A(n, n')w_k'\| \leq \frac{1}{c_1}a^{n'-n}
(\|w_k\| + \delta).
\end{equation*}
Since $w_k' \in \cV_\tx s(c_3, n')$, \eqref{eq:norm-equiv-Im} yields
\begin{equation}
\label{eq:wk'petit}
\|w_k'\| \leq \frac{1+c_3}{c_1}I_{n'}^-(w_k') \leq \frac{1+c_3}{c_1^2}a^{n'-n}(\|w_k\| + \delta).
\end{equation}

Let $w' := w_1' - w_2'$ and $\wt w := A(n, n')w'$. We have $\|\wt w - w\| \leq 2\delta$.
There are two cases: either $\wt w \in \cV_\tx u(c_3, n)$, or not.

In the first case, we also have $w' \in \cV_\tx u(c_3, n')$, so we obtain
\begin{equation*}
\|w'\| \geq c_1 I_{n'}^+(w') \geq c_1 b^{n'-n}I_n^+(\wt w) \geq \frac{c_1^2 c_3}{1+c_3}b^{n'-n}\|\wt w\|.
\end{equation*}
Combining this with \eqref{eq:wk'petit} and $\|w'\| \leq \|w_1'\| + \|w_2'\|$ we get
\begin{equation*}
\|\wt w\| \leq \frac{(1+c_3)^2}{c_1^4 c_3}\Big(\frac ab\Big)^{n'-n}(\|w_1\| + \|w_2\| + 2\delta),
\end{equation*}
which implies $\|w\| \leq \|\wt w\| + 2\delta \leq 4\delta$ by taking $n'$ large enough (depending on $\delta$).
Since $\delta > 0$ is arbitrary, this finishes the proof.

In the second case, since $X_\tx u(n) \subset \cV_\tx u(c_4, n)$, we have
\begin{equation*}
I_n^+(w) \geq c_4 I_n^-(w),\quad I_n^+(\wt w) < c_3 I_n^-(\wt w), \quad \|w - \wt w\| \leq 2\delta.
\end{equation*}
By continuity of $I_n^+, I_n^-$, since $\delta > 0$ is arbitrary, this yields $I_n^-(w) = 0$.
Again by continuity, we have $I_n^-(\wt w)$ as small as we wish, thus also $\|\wt w\|$ as small as we wish.
Hence $\|w\|$ is small as well.

This finishes the proof that $X = X_\tx s(n) \oplus X_\tx u(n)$ for all $n$.
The fact that
\begin{equation}
\label{eq:Xs-dim}
\dim X_\tx s(n) = K
\end{equation}
follows from \eqref{eq:Xu-codim}.
We are ready to verify all the requirements in Definition~\ref{def:exp-dich}.

Invertibility of $A(n, m)|_{X_\tx s(m)}: X_\tx s(m) \to X_\tx s(n)$ follows from Corollary~\ref{cor:Xs}
and \eqref{eq:Xs-dim}.

Uniform boundedness of the projections $\pi_\tx s(n)$, $\pi_\tx u(n)$ follows from \eqref{eq:v+w-lbound}.

The fact that the projections commute with $A(n, m)$ follows from $X_\tx s(n) = A_n^{-1}(X_\tx s(n-1))$
and $X_\tx u(n) = A_n^{-1}(X_\tx u(n-1))$.

If $v_n \in X_\tx s(n)$, then for any $m \geq n$ there exists $v_m = A(n, m)^{-1}v_n$.
Moreover, $v_m \in \cV_\tx s(c_3, m)$ for $m \geq n$, so \eqref{eq:intro-ineq-Im0} yields $I_m^-(v_m) \leq a^{m-n}I_n^-(v_n)$.
Hence \eqref{eq:norm-equiv-Im} yields $\|v_m\| \leq Ca^{m-n}\|v_n\|$ with $C = \frac{1+c_3}{c_1^2}$.

Similarly, one can prove that if $v_m \in X_\tx u(m)$, then $\|A(n, m)v_m\| \leq Cb^{n-m}\|v_m\|$,
with $C = \frac{1+c_4}{c_1^2c_4}$.
\end{proof}
Next, we prove below that \eqref{eq:contains-K-2} holds for $c_3$ sufficiently large if
$\alpha_{k, n}^-$ are uniformly linearly independent, by which we mean that there exists $c_6 > 0$ such that
\begin{equation}
\label{eq:alpham-indep}
\Big\|\sum_{k=1}^K b_k \alpha_{k, n}^-\Big\|_{X^*} \geq c_6 \max_{1 \leq k \leq K}|b_k|,\quad\text{for all }(b_1, \ldots, b_K) \in \bR^K.
\end{equation}
\begin{proposition}
\label{prop:Kspace}
The cone $\cV_\tx s(c_3, n)$ contains no linear subspace of dimension $K + 1$.
If \eqref{eq:alpham-indep} holds and $c_3 > \frac{2K}{c_1c_2c_6}$, then it contains a linear subspace of dimension $K$.
\end{proposition}
\begin{proof}
If $\Sigma \subset X$ is a linear subspace of dimension $K+1$, then there exists $0 \neq v \in \Sigma$
such that
\begin{equation*}
\la \alpha_{k,n}^-, v\ra = 0, \qquad\text{for all }k \in \{1, \ldots, K\},
\end{equation*}
which implies $v \notin \cV_\tx s(c, n)$, for any $c > 0$.

Now assume $c_3 > \frac{2K}{c_1c_2c_6}$.
Fix $n \geq 0$ and for $k \in \{1, \ldots, K\}$ let
$Y_k := \bigcap_{j \neq k} \ker \alpha_{j, n}^-$. We have
\begin{equation}
\label{eq:hahn-banach}
\sup_{v_0 \in Y_k,\,\|v_0\| = 1}\la \alpha_{k, n}^-, v_0\ra = \inf_{(b_j)\in \bR^K,\,b_k=1}
\Big\|\sum_{j=1}^K b_j\alpha_{j, n}^-\Big\|_{X^*}.
\end{equation}
Indeed, for all $v_0 \in Y_k$ and $(b_j) \in \bR^K$ such that $\|v_0\| = b_k = 1$ we have
\begin{equation*}
\la \alpha_{k, n}^-, v_0\ra = \Big\la \sum_{j=1}^K b_j\alpha_{j, n}^-, v_0\Big\ra
\leq \Big\|\sum_{j=1}^K b_j\alpha_{j, n}^-\Big\|_{X^*},
\end{equation*}
which implies that the left hand side in \eqref{eq:hahn-banach}
is smaller or equal to the right hand side.
Suppose the strict inequality holds, in other words $\alpha_{k, n}^-$ defines
a linear functional on $Y_k$ of norm strictly smaller than the right hand side of \eqref{eq:hahn-banach}.
Then, by the Hahn-Banach theorem, there exists $\alpha \in X^*$
such that $\|\alpha\|_{X^*}$ is strictly smaller than the right hand side of \eqref{eq:hahn-banach}
and $Y_k \subset \ker(\alpha - \alpha_{k, n}^-)$. But it is well-known that the last condition
implies that there exist $b_1, \ldots, b_{k-1}, b_{k+1}, \ldots, b_K \in \bR$ such that
$\alpha - \alpha_{k, n}^- = \sum_{j \neq k}b_j\alpha_{j, n}^-$, so we get a contradiction.
This proves \eqref{eq:hahn-banach}.

By \eqref{eq:alpham-indep}, the right hand side of \eqref{eq:hahn-banach} is $\geq c_6$,
thus for all $k \in \{1, \ldots, K\}$ there exists $z_k \in X$ such that
\begin{equation}
\label{eq:Z-properties}
\|z_k\| = 1,\quad \la\alpha_{k, n}^-, z_k\ra \geq \frac 12 c_6,\quad
\la \alpha_{j, n}^-, z_k\ra = 0\ \text{for }j \neq k.
\end{equation}
Let $\wt X_{\tx s}(n)$ be the subspace spanned by the vectors $z_k$.
Clearly, the vectors $z_k$ are linearly independent, so $\dim\wt X_{\tx s} = K$.
Let $(a_k) \in \bR^K$. We should prove that $v_0 := \sum_k a_k z_k \in \cV_{\tx s}(c_2, n)$.
From \eqref{eq:Z-properties} and \eqref{eq:Inm-form}, we have
\begin{equation*}
I_n^-(v_0) \geq c_2\max_k |\la \alpha_{k, n}^-, v_0\ra| \geq \frac 12 c_2c_6\max_k |a_k| \geq \frac{c_2c_6}{2K}\|v_0\| \geq \frac{c_1c_2c_6}{2K}I_n^+(v_0),
\end{equation*}
where the last inequality follows from \eqref{eq:norm-equiv}.
Since $c_3 > \frac{2K}{c_1c_2c_6}$, this shows that $v_0 \in \cV_\tx s(c_3, n)$ and finishes the proof.
\end{proof}
To finish this section, we prove Proposition~\ref{prop:cond-nec}.
The proof follows a well-known scheme, see~\cite{Muldowney} in the case of ODEs.
\begin{proof}[Proof of Proposition~\ref{prop:cond-nec}]
Assume \eqref{eq:intro-dynsyst-2} has an exponential dichotomy with values $a$ and $b$. For $n \geq 0$ we define
\begin{align*}
I_n^+(v_n) &:= \sup_{0 \leq m \leq n}b^{n-m}\|B(n, m)^{-1}\pi_\tx u(n)v_n\|, \\
I_n^-(v_n) &:= \sup_{m \geq n} a^{n-m}\|B(m, n)\pi_\tx s(n)v_n\|.
\end{align*}
Directly from Definition~\ref{def:intro-exp-dich-2} we get
\begin{align*}
\|\pi_\tx u(n)v_n\| \leq I_n^+(v_n) \leq C\|\pi_\tx u(n)v_n\|, \\
\|\pi_\tx s(n)v)n\| \leq I_n^-(v_n) \leq C\|\pi_\tx s(n)v_n\|,
\end{align*}
which implies \eqref{eq:norm-equiv}.
It is clear that $I_n^+$ and $I_n^-$ are seminorms, in particular they are continuous.
Moreover, we have
\begin{equation*}
\begin{aligned}
I_{n+1}^+(B_n v_n) &= \sup_{0 \leq m\leq n+1}b^{n+1-m}\|B(n+1, m)^{-1}\pi_\tx u(n+1)(B_n v_n)\| \\
&\geq b\sup_{0\leq m\leq n}b^{n-m}\|B(n+1, m)^{-1}B_{n}\pi_\tx u(n)v_n\| = b I_n^+(v_n),
\end{aligned}
\end{equation*}
and similarly $I_{n+1}^-(B_n v_n) \leq a I_n^-(v_n)$.

Now assume that $X_\tx u(n)$ has finite dimension $K$. Since $I_n^+$ is a norm on $X_\tx u(n)$,
existence of linear functionals $\alpha_{k,n}^+ \in (X_\tx u(n))^*$  such that \eqref{eq:Inp-form}
holds on $X_\tx u(n)$ (with the constant depending only on $K$) is a classical fact in Convex Geometry
(it can be proved for example using the John's ellipsoid).
Now it suffices to extend $\alpha_{k, n}^+$ on the whole $X$ be setting $\la \alpha_{k, n}^+, v\ra = 0$
for $v \in X_\tx s(n)$.
\end{proof}

\subsection{Discrete forward dynamical systems}
\label{ssec:projections-dis-2}
In this section, we prove Theorem~\ref{thm:main-dis-2}.
\begin{lemma}
\label{lem:V-inv-2}
For all $n \geq 0$ and $c \in [c_4, c_3]$ there is
\begin{gather*}
\conj{B_n \cV_\tx u(c, n)}\subset \cV_\tx u(c, n+1), \\
B_n^{-1}\cV_\tx s(c, n+1) \subset \cV_\tx s(c, n).
\end{gather*}
\end{lemma}
\begin{proof}
In order to prove the first inclusion, suppose $v_n \in X$ is such that $v_n \in \cV_\tx u(c, n)$
and $B_n v_n \notin \cV_\tx u(c, n+1)$, thus $B_n v_n \in \cV_\tx s(c, n+1)$. 
From \eqref{eq:intro-ineq-Im-2} and \eqref{eq:intro-ineq-Ip-2} we get
\begin{equation*}
I_{n+1}^+(B_n v_n) \geq b I_n^+(v_n) \geq b c I_n^-(v_n) \geq \frac{b c}{a}I_{n+1}^-(B_n v_n)
\geq c I_{n+1}^-(B_n v_n),
\end{equation*}
which contradicts $B_n v_n \notin \cV_\tx u(c, n+1)$.

In order to prove the second inclusion, suppose $v_n \in X$ is such that $B_n v_n \in \cV_\tx s(c, n+1)$
and $v_n \notin \cV_\tx s(c, n)$, thus $v_n \in \cV_\tx u(c, n)$.
From \eqref{eq:intro-ineq-Im-2} and \eqref{eq:intro-ineq-Ip-2} we get
\begin{equation*}
I_{n+1}^+(B_n v_n) \geq b I_n^+(v_n) > {b}{c}I_n^-(v_n) \geq \frac{b c}{a}I_{n+1}^-(B_n v_n)
\geq c I_{n+1}^-(B_n v_n),
\end{equation*}
which contradicts $B_n v_n \in \cV_\tx s(c, n+1)$.
\end{proof}
For $n \geq 0$ and $c \in [c_4, c_3]$ we define the stable subspace by
\begin{equation}
\label{eq:Xs-def-2}
X_\tx s(c, n) := \bigcap_{n' > n} B(n', n)^{-1}(\cV_\tx s(c, n')).
\end{equation}
Clearly, $X_\tx s(c, n)$ is a closed set.

The statement and proof of Lemma~\ref{lem:Xs-charact-2} are very similar
to the statement and proof of Lemma~\ref{lem:Xs-charact}.
We provide the details for the sake of completeness.
\begin{lemma}
\label{lem:Xs-charact-2}
For all $n \geq 0$ and $c \in [c_4, c_3]$ the following conditions are equivalent:
\begin{enumerate}[(i)]
\item
\label{it:Xs-charact-i-2}
$w \in X_\tx s(c, n)$,
\item
\label{it:Xs-charact-ii-2}
there exists $C \geq 0$ such that for all $n' \geq n$ the bound
$\|B(n', n) w\| \leq C a^{n'-n}$ holds,
\item
\label{it:Xs-charact-iii-2}
there exist $C \geq 0$ and $d < b$ such that for all $n' \geq n$ the bound
$\|B(n', n) w\| \leq C d^{n'-n}$ holds.
\end{enumerate}
\end{lemma}
\begin{proof}
\ref{it:Xs-charact-i} $\Rightarrow$ \ref{it:Xs-charact-ii}.
Let $w \in X_{\tx s}(c, n)$ and set $v_{n'} := B(n', n)w$ for $n' > n$.
By the definition of $X_\tx s(c, n)$, we have $v_{n'} \in \cV_\tx s(n')$ for all $n' > n$.
In particular, the bound \eqref{eq:norm-equiv-Im} holds.
Also, by the proof of Lemma~\ref{lem:V-inv-2}, we have $I_{m+1}^-(v_{m+1}) \leq a I_m^-(v_m)$
for all $m \geq n$. Hence, we obtain \ref{it:Xs-charact-ii-2} with $C = \frac{(1+c)\|w\|}{c_1^2}$.

\ref{it:Xs-charact-ii} $\Rightarrow$ \ref{it:Xs-charact-iii} follows from $a < b$.

\ref{it:Xs-charact-iii} $\Rightarrow$ \ref{it:Xs-charact-i}.
Let $w \in X$ be such that \ref{it:Xs-charact-iii} holds and set $v_{n'} := B(n', n) w$ for $n' > n$.
Suppose that $w \notin X_\tx s(c, n)$. This means that there exists $m > n$ such that $v_m \notin \cV_\tx s(c, m)$,
thus $v_m \in \cV_\tx u(c, m)$. By Lemma~\ref{lem:V-inv-2}, we have $v_{n'} \in \cV_\tx u(c, n')$ for all $n' \geq m$.
Thus
\begin{equation*}
\frac{C}{c_1}d^{n'-n}\geq \frac{1}{c_1}\|v_{n'}\| \geq I_{n'}^+(v_{n'}) \geq b^{n' - m} I_m^+(v_m)
\geq c_1 b^{n'-m}\|v_m\|.
\end{equation*}
Since we assume $d < b$, this implies $\|v_m\| = 0$, contradicting $v_m \notin \cV_\tx s(c, m)$.
\end{proof}
\begin{corollary}
\label{cor:Xs-2}
For all $n \geq 0$, $X_\tx s(c, n) = X_\tx s(n)$ does not depend on $c \in [c_4, c_3]$.
It is a closed linear subspace of $X$
and $X_\tx s(n) = B_n^{-1}X_\tx s(n+1)$.
\end{corollary}
\begin{proof}
Condition~\ref{it:Xs-charact-ii-2} in Lemma~\ref{lem:Xs-charact-2} defines a linear subspace
independent of $c \in [c_4, c_3]$ so, by Lemma~\ref{lem:Xs-charact-2}, $X_\tx s(n)$ is a linear subspace of $X$.
We see directly from \eqref{eq:Xs-def-2} that it is closed
and that $X_\tx s(n) = B_n^{-1}X_\tx s(n+1)$.
\end{proof}
\begin{proof}[Proof of Theorem~\ref{thm:main-dis-2}]
We let $X_\tx u(0)$ be any $K$-dimensional linear subspace contained in $\cV_\tx u(c_3, 0)$.
Such a space exists by assumption~\eqref{eq:contains-K-2}.
If $v_0 \in X_\tx u(0)$ and $v_{n+1} = B_n v_{n}$ for $n \geq 0$, then by Lemma~\ref{lem:V-inv-2}
$v_n \in \cV_\tx u(c_3, n)$ for all $n$, which by \eqref{eq:intro-ineq-Ip-20} yields
$I_n^+(v_n) \geq b^n I_0^+(v_0)$. Using \eqref{eq:norm-equiv-Ip}, we obtain
\begin{equation*}
\|v_n\| \geq \frac{c_1^2 c_3}{1+c_3}b^n\|v_0\|,
\end{equation*}
thus $B(n, 0)\vert_{X_\tx u(0)}$ is a linear embedding. We set
\begin{equation*}
X_{\tx{u}}(n) := B(n, 0)X_\tx u(0),
\end{equation*}
which is a linear subspace of $X$ of dimension $K$.
As for backward systems, the choice of $X_\tx u(n)$ is not canonical.

We will find $c_5 > 0$ such that if $v \in X_\tx s(n)$ and $w \in X_\tx u(n)$,
then
\begin{equation}
\label{eq:v+w-lbound-2}
\|v + w\| \geq c_5\|w\|.
\end{equation}
We can assume $\|v\| \leq \frac 32 \|w\|$. 
Since $X_\tx u(n) \subset \cV_\tx u(c_3, n)$, \eqref{eq:norm-equiv-Ip} yields $I_n^+(w) \geq \frac{c_1c_3}{1+c_3}\|w\|_X$,
thus by \eqref{eq:Inp-form} there is $k_0 \in \{1, \ldots, K\}$ such that
\begin{equation*}
|\la \alpha_{k_0}^+, w\ra| \geq \frac{c_1c_2c_3}{1+c_3}\|w\|.
\end{equation*}
Since $X_\tx s(n) \subset \cV_\tx s(c_4, n)$, we have $I_n^+(v) \leq c_4 I_n^-(v)$,
so \eqref{eq:norm-equiv} yields $(1+c_4)I_n^+(v) \leq \frac{c_4}{c_1}\|v\| \leq \frac{3c_4}{2c_1}\|w\|$.
Invoking again \eqref{eq:Inp-form} we obtain
\begin{equation*}
|\la \alpha_{k_0}^+, v\ra| \leq \frac{3c_4}{2c_1c_2(1+c_4)}\|w\|.
\end{equation*}
From \eqref{eq:alphav-big} and \eqref{eq:alphaw-small} we get
\begin{equation*}
\|v + w\| \geq c_1 I_n^+(v+w) \geq c_1c_2|\la \alpha_{k_0}^+, v+w\ra| \geq c_1c_2\Big(\frac{c_1c_2c_3}{1+c_3} - \frac{3c_4}{2c_1c_2(1+c_4)}\Big)\|w\|.
\end{equation*}
Assumption \eqref{eq:intro-c3c4-2} implies that the constant in front of $\|w\|$ is $> 0$, so we have proved \eqref{eq:v+w-lbound-2}.

As in the proof of Theorem~\ref{thm:main-dis}, we obtain $X_\tx s(n) \cap X_\tx u(n) = \{0\}$ for all $n \geq 0$.
Let $u \in X$ and let $\Pi := u + X_\tx u(n)$.
For any $n' \geq n$, let $\wt X_\tx s(n') := \bigcap_{k=1}^K \ker(\alpha_{k, n'}^-)$.
We see that $\codim(\wt X_\tx s(n')) \leq K$ and $\wt X_\tx s(n') \subset \cV_\tx s(n')$,
which implies that $B(n', n)^{-1}(\cV_{\tx s}(n'))$ contains a space of codimension $K$.
Since $X_\tx u(n) \cap B(n', n)^{-1}(\cV_{\tx s}(n')) = \{0\}$, we obtain $\Pi \cap B(n', n)^{-1}(\cV_{\tx s}(n')) \neq \emptyset$.
Thus $\Pi \cap B(n', n)^{-1}(\cV_\tx s(n'))$ is a nested family of closed non-empty sets
and it suffices to show that their diameters tend to 0 as $n' \to \infty$,
which can be done similarly as in the proof of Theorem~\ref{thm:main-dis}.
\end{proof}
If $\alpha_{k, n}^-$ are uniformly linearly independent, by which we mean that there exists $c_6 > 0$ such that
\begin{equation}
\label{eq:alphap-indep}
\Big\|\sum_{k=1}^K b_k \alpha_{k, n}^+\Big\|_{X^*} \geq c_6 \max_{1 \leq k \leq K}|b_k|,\quad\text{for all }(b_1, \ldots, b_K) \in \bR^K,
\end{equation}
then the proof of Proposition~\ref{prop:Kspace} shows that \eqref{eq:contains-K} holds if $c_3$ is small enough. We have
\begin{proposition}
\label{prop:Kspace-2}
The cone $\cV_\tx u(c_3, n)$ contains no linear subspace of dimension $K + 1$.
If \eqref{eq:alphap-indep} holds and $c_3 < \frac{c_1c_2c_6}{2K}$, then it contains a linear subspace of dimension $K$. \qed
\end{proposition}
The proof is exactly the same as the proof of Proposition~\ref{prop:Kspace}, so we skip it.
\subsection{Strongly continuous backward dynamical systems}
\label{ssec:projections-cont}
Our proof adapts easily to the case of continuous dynamics.
\begin{definition}
\label{def:evol-op-back}
Let $X$ be a Banach space. A family of operators $S(t, \tau) \in \scrL(X)$ for $0 \leq t \leq \tau$
is called a \emph{strongly continuous backward evolution operator} if it satisfies:
\begin{enumerate}[(1)]
\item $S(t, t) = \Id$ for all $t \geq 0$,
\item for all $\tau \geq 0$ and all $v_{\tau} \in X$ the function $[0, \tau] \owns t \mapsto S(t, \tau)v_{\tau} \in X$ is continuous,
\item for all $0 \leq t \leq s \leq \tau$ there is $S(t, s)\circ S(s, \tau) = S(t, \tau)$.
\end{enumerate}
\end{definition}
Let $X$ be a Banach space and let $S(t, \tau)$ for $0 \leq t \leq \tau$ be a strongly continuous backward evolution operator.
We consider the dynamical system
\begin{equation}
\label{eq:dynsyst-cont}
v(t) = S(t, \tau)v(\tau), \qquad v(0) = v_0 \in X.
\end{equation}
Note that we do not require $S(t, \tau)$ to be invertible.
\begin{definition}
\label{def:exp-dich-cont}
We say that \eqref{eq:dynsyst} has a (uniform) \emph{exponential dichotomy} with exponents $\lambda$ and $\mu$,
$-\infty < \lambda < \mu < \infty$,
if for all $t \geq 0$ there exists a direct sum decomposition $X = X_\tx s(t) \oplus X_\tx u(t)$
such that $X_\tx s(t)$, $X_\tx u(t)$ and the associated projections
$\pi_\tx s(t): X \to X_\tx s(t)$ and $\pi_\tx u(t): X \to X_\tx u(t)$ have the following properties for all $t \leq \tau$:
\begin{enumerate}
\item $S(t, \tau)\circ\pi_\tx s(\tau) = \pi_\tx s(t)\circ S(t, \tau)$ and $S(t, \tau)\circ\pi_\tx u(\tau) = \pi_\tx u(t)\circ S(t, \tau)$,
\item there exists a constant $C$ such that $\|\pi_\tx s(t)\|_{\scrL(X)} + \|\pi_\tx u(t)\|_{\scrL(X)} \leq C$,
\item $S(t, \tau)\vert_{X_\tx s(\tau)}: X_\tx s(\tau) \to X_\tx s(t)$ is invertible,
\item there exists a constant $C$ such that $\|S(t, \tau)^{-1}v_t\| \leq C \eee^{\lambda(\tau - t)}\|v_\tau\|$ for all $v_t \in X_\tx s(t)$,
\item there exists a constant $C$ such that $\|S(t, \tau)v_\tau\| \leq C \eee^{\mu(t - \tau)}\|v_\tau\|$ for all $v_\tau \in X_\tx u(\tau)$.
\end{enumerate}
\end{definition}
Our sufficient condition for existence of an exponential dichotomy is expressed in terms
of two families of (nonlinear) homogeneous functionals $I_t^-, I_t^+ : X \to \bR_+$.
We assume that $I_t^\pm(v)$ is continuous in $(t, v) \in \bR_+ \times X$.
Given $I_t^-$, $I_t^+$ and a number $c > 0$, we define the stable and the unstable cone
\begin{gather*}
\cV_\tx s(c, t) := \{v \in X: I_t^+(v) \leq c I_t^-(v)\}, \label{eq:Vs-def-c} \\
\cV_\tx u(c, t) := \{v \in X: I_t^+(v) \geq c I_t^-(v)\}. \label{eq:Vu-def-c}
\end{gather*}

Firstly, we assume that there exists $c_1 > 0$ (independent of $t$) such that
\begin{equation}
\label{eq:norm-equiv-c}
c_1 \|v\|_X \leq I_t^-(v) + I_t^+(v) \leq \frac{1}{c_1}\|v\|_X,
\qquad \text{for all }t \geq 0 \text{ and }v \in X.
\end{equation}\noeqref{eq:norm-equiv-c}

Secondly, we assume that there exist $c_2 > 0$, $K \in \{0, 1, 2, \ldots\}$
and $\alpha_{k, t}^- \in X^*$ for $(k, t) \in \{1, \ldots, K\}\times \bR_+$ such that
\begin{equation}
\label{eq:Inm-form-c}
c_2 \max_{1 \leq k \leq K}|\la \alpha_{k, t}^-, v\ra| \leq I_t^-(v) \leq \frac{1}{c_2}\max_{1 \leq k \leq K}|\la \alpha_{k, t}^-, v\ra|.
\end{equation}\noeqref{eq:Inm-form-c}
Continuity of $\alpha_{k, t}^-$ with respect to $t$ is not required.

Lastly, 
we assume that there exist $c_3, c_4 > 0$ and $-\infty < \lambda < \mu < \infty$ such that
\begin{align}
&V_\tx s(c_4, t)\text{ contains a linear space of dimension }K\text{ for all }t, \label{eq:contains-K-c} \\
&c_4 > 3(c_1c_2)^{-2}(c_3+1), \label{eq:intro-c3c4-c} \\
&I_{\tau}^-(v_\tau) \leq \eee^{\lambda(\tau - t)} I_t^-(S(t, \tau) v_\tau)\quad\text{if }S(t', \tau)v_\tau \in \cV_\tx s(c_4, t')\text{ for all }t' \in [t, \tau], \label{eq:ineq-Im-c} \\
&I_{\tau}^+(v_\tau) \geq \eee^{\mu(\tau - t)} I_{t}^+(S(t, \tau)v_\tau)\quad\text{if }S(t', \tau)v_\tau \in \cV_\tx u(c_3, t')\text{ for all }t' \in [t, \tau]. \label{eq:ineq-Ip-c}
\end{align}
\noeqref{eq:contains-K-c, eq:intro-c3c4-c}
Note that, unlike in \eqref{eq:intro-ineq-Im} and \eqref{eq:intro-ineq-Ip}, here we assume that the $I_t^-$ or $I_t^+$
direction is significant on the whole time interval $[t, \tau]$.
This is why the proof of the invariance of cones given below contains an additional continuity argument.
\begin{lemma}
\label{lem:V-inv-cont}
For all $c \in (c_3, c_4)$ and $0 \leq t \leq \tau$ there is
\begin{gather*}
S(t, \tau)^{-1} \cV_\tx u(c, t)\subset \cV_\tx u(c, \tau), \\
\conj{S(t, \tau)\cV_\tx s(c, \tau)} \subset \cV_\tx s(c, t).
\end{gather*}
\end{lemma}
\begin{proof}
In order to prove the first inclusion, suppose $v_\tau \in X$ is such that $S(t, \tau) v_\tau \in \cV_\tx u(c, t)$
and $v_\tau \notin \cV_\tx u(c, \tau)$. Set
\begin{equation*}
\label{eq:t'-def}
t_1 := \sup\{t' \leq \tau: S(t', \tau) v_\tau \in \cV_\tx u(c, t')\}.
\end{equation*}
By continuity,
\begin{gather}
I_{t_1}^+(S(t_1, \tau)v_\tau) = cI_{t_1}^-(S(t_1, \tau)v_\tau), \label{eq:I-equal-c} \\
S(t', \tau)v_\tau \in \cV_\tx s(c, t'),\quad \text{for all }t' \in [t_1, \tau]. \label{eq:in-Vs-c}
\end{gather}
Assumption \eqref{eq:ineq-Im-c} together with \eqref{eq:in-Vs-c} yields
\begin{equation*}
I_\tau^-(v_\tau) \leq \eee^{\lambda(\tau - t_1)}I_{t_1}^-(S(t_1, \tau)v_\tau).
\end{equation*}
In particular, since $I_\tau^-(v_\tau) > 0$, we have $I_{t_1}^-(S(t_1, \tau)v_\tau) > 0$.
Thus, again by continuity, \eqref{eq:I-equal-c} and $ c > c_3$ imply that there exists $t_2 \in (t_1, \tau]$ such that
\begin{equation*}
S(t', \tau)v_\tau \in \cV_\tx u(c_3, t'),\quad\text{for all }t' \in [t_1, t_2].
\end{equation*}
Let $v_{t_2} := S(t_2, \tau)v_\tau$ and $v_{t_1} := S(t_1, \tau)v_\tau = S(t_1, t_2)v_2$.
Using \eqref{eq:ineq-Ip-c} with $\tau = t_2$ and $t = t_1$ we get
\begin{equation*}
I_{t_2}^+(v_{t_2}) \geq \eee^{\mu(t_2 - t_1)}I_{t_1}^+(v_{t_1}).
\end{equation*}
On the other hand, \eqref{eq:ineq-Im-c} and \eqref{eq:in-Vs-c} yield
\begin{equation*}
I_{t_2}^-(v_{t_2}) \leq \eee^{\lambda(t_2 - t_1)}I_{t_1}^-(v_{t_1}).
\end{equation*}
Thus \eqref{eq:I-equal-c} and $\lambda < \mu$ yield $I_{t_2}^+(v_{t_2}) > c I_{t_2}^-(v_{t_2})$,
which contradicts \eqref{eq:in-Vs-c}.

In order to prove the second inclusion, suppose $v_\tau \in X$ is such that $v_\tau \in \cV_\tx s(c, \tau)$
and $S(t, \tau)v_\tau \notin \cV_\tx s(c, t)$. Set
\begin{equation*}
t_2 := \inf\{t' \geq t: S(t', \tau) v_\tau \in \cV_\tx s(c, t')\}.
\end{equation*}
By continuity,
\begin{gather}
I_{t_2}^+(S(t_2, \tau)v_\tau) = cI_{t_2}^-(S(t_2, \tau)v_\tau), \label{eq:I-equal-c-2} \\
S(t', \tau)v_\tau \in \cV_\tx u(c, t'),\quad \text{for all }t' \in [t, t_2]. \label{eq:in-Vu-c}
\end{gather}
Assumption \eqref{eq:ineq-Ip-c} together with \eqref{eq:in-Vu-c} yields
\begin{equation*}
I_\tau^+(v_{t_2}) \geq \eee^{\mu(t_2 - t)}I_{t}^+(S(t, \tau)v_\tau).
\end{equation*}
In particular, since $I_t^+(S(t, \tau)v_t) > 0$, we have $I_{t_2}^+(S(t_2, \tau)v_\tau) > 0$.
Thus, again by continuity, \eqref{eq:I-equal-c-2} and $ c < c_4$ imply that there exists $t_1 \in [t, t_2)$ such that
\begin{equation*}
S(t', \tau)v_\tau \in \cV_\tx s(c_4, t'),\quad\text{for all }t' \in [t_1, t_2].
\end{equation*}
The remaining arguments are the same as in the first part of the proof.
\end{proof}
\begin{mainthm}
\label{thm:main-cont-1}
Under assumptions \eqref{eq:norm-equiv-c}--\eqref{eq:ineq-Ip-c},
the system \eqref{eq:dynsyst-cont} has an exponential dichotomy with exponents $\lambda$ and $\mu$.
For all $t \geq 0$ the stable subspace $X_\tx s(t)$ is contained in $\cV_\tx s(c_3, t)$ and has dimension $K$.
\end{mainthm}
The proof would follow the lines of the proof of Theorem~\ref{thm:main-dis},
with $c_3$ and $c_4$ replaced everywhere by some $\wt c_3 > c_3$ and $\wt c_4 < c_4$ such that
$\wt c_4 > 3(c_1 c_2)^{-2}(\wt c_3 + 1)$.
At the end we obtain $X_\tx s(t) \subset \cV_\tx s(\wt c_3, t)$ for any $\wt c_3 > c_3$,
which means $X_\tx s(t) \subset \cV_\tx s(c_3, n)$.

\subsection{Strongly continuous forward dynamical systems}
\label{ssec:projections-cont-2}
\begin{definition}
\label{def:evol-op}
Let $X$ be a Banach space. A family of operators $T(\tau, t) \in \scrL(X)$ for $0 \leq t \leq \tau$
is called a \emph{strongly continuous evolution operator} if it satisfies:
\begin{enumerate}[(1)]
\item $T(t, t) = \Id$ for all $t \geq 0$,
\item for all $t \geq 0$ and all $v_t \in X$ the function $[t, \infty) \owns \tau \mapsto T(\tau, t)v_{t} \in X$ is continuous,
\item for all $0 \leq t \leq s \leq \tau$ there is $T(\tau, s)\circ T(s, t) = T(\tau, t)$.
\end{enumerate}
\end{definition}
Let $T(\tau, t)$ be a strongly continuous evolution operator and consider the system
\begin{equation}
\label{eq:dynsyst-2-cont}
v_\tau = T(\tau, t) v_t, \qquad v_0 \in X.
\end{equation}
\begin{definition}
\label{def:exp-dich-cont-2}
We say that \eqref{eq:dynsyst-2-cont} has an \emph{exponential dichotomy} with exponents $\lambda$ and $\mu$,
$-\infty < \lambda < \mu < \infty$,
if for all $t \geq 0$ there exists a direct sum decomposition $X = X_\tx s(t) \oplus X_\tx u(t)$
such that $X_\tx s(t)$, $X_\tx u(t)$ and the associated projections
$\pi_\tx s(t): X \to X_\tx s(t)$ and $\pi_\tx u(t): X \to X_\tx u(t)$ have the following properties for all $t \leq \tau$:
\begin{enumerate}
\item $T(\tau, t)\circ\pi_\tx s(t) = \pi_\tx s(\tau)\circ T(\tau, t)$ and $T(\tau, t)\circ\pi_\tx u(t) = \pi_\tx u(\tau)\circ T(\tau, t)$,
\item there exists a constant $C$ such that $\|\pi_\tx s(t)\|_{\scrL(X)} + \|\pi_\tx u(t)\|_{\scrL(X)} \leq C$,
\item $T(\tau, t)\vert_{X_\tx u(t)}: X_\tx u(t) \to X_\tx u(\tau)$ is invertible,
\item there exists a constant $C$ such that $\|T(\tau, t)v_t\| \leq C \eee^{\lambda(\tau - t)}\|v_t\|$ for all $v_t \in X_\tx s(t)$,
\item there exists a constant $C$ such that $\|T(\tau, t)^{-1}v_\tau\| \leq C \eee^{\mu(t-\tau)}\|v_\tau\|$ for all $v_\tau \in X_\tx u(\tau)$.
\end{enumerate}
\end{definition}

Our sufficient conditions for existence of an exponential dichotomy are similar as in Section~\ref{ssec:projections-cont}.
Instead of \eqref{eq:Inm-form-c}, we assume
\begin{equation}
\label{eq:Inp-form-c}
c_2 \max_{1 \leq k \leq K}|\la \alpha_{k, t}^+, v\ra| \leq I_t^+(v) \leq \frac{1}{c_2}\max_{1 \leq k \leq K}|\la \alpha_{k, t}^+, v\ra|.
\end{equation}\noeqref{eq:Inp-form-c}
Instead of \eqref{eq:contains-K-c}--\eqref{eq:ineq-Ip-c}, we assume
\begin{align}
&V_\tx u(c_3, t)\text{ contains a linear space of dimension }K\text{ for all }t, \label{eq:contains-K-c-2} \\
&c_4 < \frac 13(c_1c_2)^{2}\frac{c_3}{1+c_3}, \label{eq:intro-c3c4-c-2} \\
&I_{\tau}^-(T(\tau, t)v_t) \leq \eee^{\lambda(\tau - t)} I_t^-(v_t)\quad\text{if }T(t', t)v_t \in \cV_\tx s(c_3, t')\text{ for all }t' \in [t, \tau], \label{eq:ineq-Im-c-2} \\
&I_{\tau}^+(T(\tau, t)v_t) \geq \eee^{\mu(\tau - t)} I_{t}^+(v_t)\quad\text{if }T(t', t)v_t \in \cV_\tx u(c_4, t')\text{ for all }t' \in [t, \tau]. \label{eq:ineq-Ip-c-2}
\end{align}
\noeqref{eq:contains-K-c-2, eq:intro-c3c4-c-2, eq:ineq-Im-c-2, eq:ineq-Ip-c-2}

\begin{mainthm}
\label{thm:main-cont-2}
Under assumptions \eqref{eq:norm-equiv-c} and \eqref{eq:Inp-form-c}--\eqref{eq:ineq-Ip-c-2},
the system \eqref{eq:dynsyst-2-cont} has an exponential dichotomy with exponents $\lambda$ and $\mu$.
For all $t \geq 0$ the stable space $X_\tx s(t)$ is contained in $\cV_\tx s(c_4, t)$ and has codimension $K$.
\qed
\end{mainthm}

\section{Some simple examples}
\label{sec:examples}

\subsection{Avalanche dynamics in finite dimension}
\label{ssec:avalanche}
Let $B_n$ be a sequence of real matrices of size $d \in \{1, 2, \ldots\}$.
We consider the linear dynamical system
\begin{equation}
\label{eq:avalanche}
v_{n+1} = B_n v_n, \qquad v_0 \in \bR^d.
\end{equation}
Assume that there exist $0\leq a < b$ such that for all $n$ the matrix $B_n^* B_n$ has $d_\tx s$ eigenvalues $\leq a^2$
and $d_\tx u = d - d_\tx s$ eigenvalues $\geq b^2$.

Let $Y_\tx s(n) \subset \bR^d$ be the subspace spanned by eigenvectors of $B_n^* B_n$
corresponding to eigenvalues $\leq a^2$,
and let $Y_\tx u(n) \subset \bR^d$ be spanned by eigenvectors of $B_n^* B_n$
corresponding to eigenvalues $\geq b^2$.
The standard formula $\|B_n v\|^2 = \la v, B_n^* B_n v\ra$, together with the Spectral Theorem for symmetric matrices yields
\begin{gather*}
v \in Y_\tx s(n) \Rightarrow \|B_n v\| \leq a \|v\|, \\
v \in Y_\tx u(n) \Rightarrow \|B_n v\| \geq b \|v\|.
\end{gather*}

Let $Z_\tx s(n) := B_n Y_\tx s(n)$ and $Z_\tx u(n) := B_n Y_\tx u(n)$.
Note that we do not assume that $B_n$ is invertible, so it may happen that $\dim(Z_\tx s(n)) < \dim(Y_\tx s(n))$.
It turns out that $\dim(X_\tx s) = d_\tx s$ if the angles between $Z_\tx s(n)$ and $Y_\tx s(n+1)$,
as well as the angles between $Z_\tx u(n)$ and $Y_\tx u(n+1)$, are small.
A~similar assumption appears in the ``Avalanche Principle'' of Goldstein and Schlag~\cite{GS01}.

If $Y, Z \subset \bR^d$ are two linear subspaces of the same dimension,
we can measure their proximity by the Hausdorff distance of their unit spheres
$S_Y := \{v \in Y: \|v\| = 1\}$ and $S_Z := \{w \in Z: \|w\| = 1\}$. We set
\begin{equation*}
\phi(Y, Z) := \max\Big(\sup_{v \in S_Y}\inf_{w \in S_Z}\|v - w\|^2, \sup_{w \in S_Z}\inf_{v \in S_Y}\|v - w\|^2\Big).
\end{equation*}
Note that $\phi(Y, Z) = 0$ if and only if $Y = Z$.

\begin{lemma}
\label{lem:length-vect}
Assume $\phi(Y, Z) \leq \delta$. Let $v \in Y$ and let $\wt v \in Z$ be the orthogonal projection of $v$ on $Z$.
Then
\begin{equation*}
\label{eq:length-vect}
\|\wt v\|^2 \geq (1 - \delta)\|v\|^2.
\end{equation*}
\end{lemma}
\begin{proof}
By rescaling, we can assume $\|v\| = 1$.
By compactness and the definition of $\phi(Y, Z)$,
there exists $w \in Z$ such that $\|v - w\|^2 \leq \delta$.
This implies $\|v - \wt v\|^2 \leq \delta$, thus $\|\wt v\|^2 = \|v\|^2 - \|v - \wt v\|^2 \geq 1 - \delta$.
\end{proof}

We have $X = Y_\tx s(n) \oplus Y_\tx u(n)$ and $Y_\tx s(n)$ is orthogonal to $Y_\tx u(n)$.
Since $Y_\tx s(n)$ and $Y_\tx u(n)$ are invariant for $B_n^* B_n$, for $v \in Y_\tx s(n)$ and $w \in Y_\tx u(n)$
we obtain
$\la B_n v, B_n w \ra = \la v, B_n^* B_n w \ra = 0$,
so $Z_\tx s(n)$ and $Z_\tx u(n)$ are orthogonal as well (but do not have to span $\bR^d$).
\begin{proposition}
\label{prop:aval}
For any $0 \leq a < b$ and $\epsilon > 0$ there exists $\delta > 0$ such that if $\phi(Z_\tx s(n), Y_\tx s(n+1)) \leq \delta$ and $\phi(Z_\tx u(n), Y_\tx u(n+1)) \leq \delta$
for all $n$, then the system \eqref{eq:avalanche} has an exponential splitting $\bR^d = X_\tx s(n) \oplus X_\tx u(n)$
with values $a + \epsilon$ and $b - \epsilon$. Moreover, $\dim(X_\tx s(n)) = d_\tx s$ and $\dim(X_\tx u(n)) = d_\tx u$.
\end{proposition}
\begin{proof}
We apply Theorem~\ref{thm:main-dis-2}.
Let $\alpha_{1, n}^-, \ldots, \alpha_{d_\tx s, n}^-$ be an orthonormal basis of $Y_\tx s(n)$
and let $\alpha_{1, n}^+, \ldots, \alpha_{d_\tx u, n}^+$ be an orthonormal basis of $Y_\tx u(n)$.
We define
\begin{equation*}
\label{eq:ImIp-avalanche}
I_n^-(v) := \sqrt{\sum_{k=1}^{d_\tx s}\la \alpha_{k, n}^-, v\ra^2}, \quad
I_n^+(v) := \sqrt{\sum_{k=1}^{d_\tx u}\la \alpha_{k, n}^+, v\ra^2},
\end{equation*}
so that $I_n^-(v)$ and $I_n^+(n)$ are the lengths of the orthogonal projections of $v$ on $Y_\tx s(n)$
and $Y_\tx u(n)$.
We have to check \eqref{eq:intro-ineq-Im-2} and \eqref{eq:intro-ineq-Ip-2} for some $c_4 \ll c_3$
(all the other conditions are immediate).

Let $v = v_\tx s + v_\tx u$ with $v_\tx s \in Y_\tx s(n)$ and $v_\tx u \in Y_\tx u(n)$.
Then $w = B_n v = w_\tx s + w_\tx u$, where $w_\tx s = B_n v_\tx s \in Z_\tx s(n)$
and $w_\tx u = B_n v_\tx u \in Z_\tx u(n)$.
We have $\|w_\tx s\| \leq a\|v_\tx s\|$ and $\|w_\tx u\| \geq b\|v_\tx u\|$.
We decompose further $w_\tx s = w_\tx{ss} + w_\tx{su}$ and $w_\tx u = w_\tx{us} + w_\tx{uu}$,
with $w_\tx{ss}, w_\tx{us} \in Y_\tx s(n+1)$ and $w_\tx{su}, w_\tx{uu} \in Y_\tx u(n+1)$.
Note that $\|w_\tx s\|^2 = \|w_\tx{ss}\|^2 + \|w_\tx{su}\|^2$, $\|w_\tx u\|^2 = \|w_\tx{us}\|^2 + \|w_\tx{uu}\|^2$,
$I_{n+1}^-(w) = \|w_\tx{ss} + w_\tx{us}\|$ and $I_{n+1}^+(w) = \|w_\tx{su} + w_\tx{uu}\|$.
We need to show that
\begin{align}
\label{eq:aval-1}
\|v_\tx u\| \geq c_4\|v_\tx s\|&\Rightarrow \|w_\tx{su} + w_\tx{uu}\| \geq (b-\epsilon)\|v_\tx u\|, \\
\label{eq:aval-2}
\|w_\tx u\| \leq c_3\|w_\tx{ss} + w_\tx{us}\|&\Rightarrow\|w_\tx{ss} + w_\tx{us}\| \leq (a + \epsilon)\|v_\tx s\|.
\end{align}

In order to prove \eqref{eq:aval-1}, we observe that Lemma~\ref{lem:length-vect} yields $\|w_\tx{uu}\| \geq \sqrt{1-\delta}\|w_\tx u\|$, thus
\begin{equation}
\label{eq:aval-3}
\|w_\tx{uu}\| \geq b\sqrt{1-\delta}\|v_\tx u\|.
\end{equation}
Using again Lemma~\ref{lem:length-vect}, we obtain
\begin{equation}
\label{eq:aval-4}
\|w_\tx{su}\|^2 = \|w_\tx s\|^2 - \|w_\tx{ss}\|^2 \leq \delta \|w_\tx s\|^2 \leq \delta a^2 \|v_\tx s\|^2
 \leq \frac{\delta a^2}{c_4^2}\|v_\tx u\|^2,
\end{equation}
where in the last step we use the assumption $\|v_\tx u\| \geq c_4 \|v_\tx s\|$.
Combining \eqref{eq:aval-3} and \eqref{eq:aval-4} we get
\begin{equation*}
\|w_\tx{su} + w_\tx{uu}\| \geq \|w_\tx{uu}\| - \|w_\tx{su}\| \geq \Big(b\sqrt{1-\delta} - \frac{a\sqrt\delta}{c_4}\Big)\|v_\tx u\|,
\end{equation*}
which is \eqref{eq:aval-1} with $\epsilon = b(1 - \sqrt{1-\delta}) + \frac{a\sqrt\delta}{c_4}$.

We are left with \eqref{eq:aval-2}. We have $\|w_\tx{ss}\| \leq \|w_\tx s\| \leq a\|v_\tx s\|$
and, similarly as in \eqref{eq:aval-4},
\begin{equation*}
\|w_\tx{us}\| \leq \delta\|w_\tx u\| \leq \frac{\delta}{c_3}\|w_\tx{ss} + w_\tx{us}\|,
\end{equation*}
which yields
\begin{equation*}
\|w_\tx{ss}+w_\tx{us}\| \leq \Big(1 - \frac{\delta}{c_3}\Big)^{-1}\|w_\tx{ss}\| \leq a\Big(1 - \frac{\delta}{c_3}\Big)^{-1}\|v_\tx{s}\|.
\end{equation*}
This proves \eqref{eq:aval-2} with $\epsilon = a\big(\big(1 - \frac{\delta}{c_3}\big)^{-1} - 1\big)$.
\end{proof}
\begin{remark}
We see that for any $c_3, c_4 > 0$ all the conditions are satisfied if $\delta$ is small enough.
In particular, taking $c_4$ small enough, we deduce from Theorem~\ref{thm:main-dis-2} that $\phi(X_\tx s(n), Y_\tx s(n)) \to 0$
as $\delta \to 0$.
\end{remark}
\subsection{Backward heat equation with an almost constant potential}
\label{ssec:backward-heat}
As our next example, we consider the backward heat equation with a time-dependent potential:
\begin{equation}
\label{eq:back-heat}
\begin{gathered}
\partial_t u(t, x) = {-}\Delta u(t, x) + V(t, x)u(t, x), \qquad \text{for }(t, x) \in \bR_+\times \Omega. \\
\end{gathered}
\end{equation}
We assume that $\Omega \subset \bR^d$ is bounded with smooth boundary and that
$V \in L^\infty([0, \infty); L^p(\Omega))$ for some $p > \frac d2$.
To simplify, we will also assume $d \geq 4$, but straightforward modifications allow to cover $d = 1, 2, 3$ as well
(in this case, one should take $p = 2$).
Given a potential $V \in L^p(\Omega)$, we denote $\lambda_j(V)$ the $j$-th smallest eigenvalue
(counted with multiplicities) of the Schr\"odinger operator $-\Delta + V$ with Dirichlet boundary conditions.

We assume that there exists $\mu > 0$ such that for all $t \in [0, \infty)$
\begin{equation*}
\label{eq:back-heat-spectral-gap}
\lambda_1(V(t)) \leq -\mu,\qquad \lambda_2(V(t)) \geq \mu.
\end{equation*}
Note that, upon adding a fixed constant to the potential,
we could cover the case where $[\lambda_1(V(t)), \lambda_2(V(t))]$ contains a given interval of strictly positive length for all $t$.
\begin{proposition}
\label{prop:back-heat}
For any $\epsilon > 0$ there exists $\delta = \delta(\Omega, p, \|V\|_{L^\infty L^p}, \mu, \epsilon) > 0$ such that if
\begin{equation}
\label{eq:back-heat-ass}
\|V(t_1) - V(t_2)\|_{L^p} \leq \delta\qquad\text{for all }t_1, t_2\text{ with }|t_1 - t_2| \leq 1,
\end{equation}
then there exists a unique (up to multiplying by a constant) non-trivial solution $u_\tx s(t): [0, \infty) \to H_0^1(\Omega)$
of \eqref{eq:back-heat} satisfying
\begin{equation*}
\sup_{t\geq 0}\eee^{-(\mu - \epsilon)t}\|u_\tx s(t)\|_{H_0^1} < \infty.
\end{equation*}
In addition, this solution satisfies
\begin{equation*}
\sup_{t\geq 0}\eee^{(\mu - \epsilon)t}\|u_\tx s(t)\|_{H_0^1} < \infty.
\end{equation*}
\end{proposition}
\begin{remark}
Condition \eqref{eq:back-heat-ass} means that on any time interval of unit length the potential,
though potentially highly oscillatory, is close in $L^p$ to a fixed function.
Note however that on large time intervals the potential can change considerably.
\end{remark}
Before giving a proof, recall a few elementary facts from Spectral Theory.
For a given potential $V \in L^p$, we denote $\phi_1(V)$ the positive eigenfunction corresponding to the smallest
eigenvalue $\lambda_1(V)$, normalised so that $\|\phi_1(V)\|_{L^2} = 1$.
\begin{lemma}
\label{lem:back-heat-spectral}
For any $M \geq 0$ there exists $C = C(\Omega, p, M, \mu) \geq 0$ such that for all $V, W$ with $\|V\|_{L^p}, \|W\|_{L^p} \leq M$,
$\lambda_1(V) \leq -\mu$, $\lambda_2(V) \geq \mu$ the following bounds hold:
\begin{gather}
|\lambda_1(V)| \leq C, \qquad \|\phi_1(V)\|_{H_0^1} \leq C, \label{eq:lambda1-bound} \\
|\lambda_1(V) - \lambda_1(W)| \leq C\|V - W\|_{L^p}, \label{eq:lambda12-bound} \\
\|\phi_1(V) - \phi_1(W)\|_{H_0^1} \leq C\sqrt{\|V - W\|_{L^p}}, \label{eq:phi12-bound} \\
\la u, ({-}\Delta + V)u\ra \geq \frac 1C \|u\|_{H_0^1}^2 - C\la \phi_1(V), u\ra^2, \quad \forall u \in H_0^1(\Omega), 
\label{eq:back-heat-coer} \\
\|({-}\Delta + V)u\|_{L^2}^2 \geq \frac 1C \|{-}\Delta u\|_{L^2}^2  - C\la \phi_1(V), u\ra^2, \quad \forall u \in H_0^1(\Omega)\cap H^2(\Omega),
\label{eq:back-heat-H2coer} \\
\|({-}\Delta + V)u\|_{L^2}^2 \geq \mu \la u, ({-}\Delta + V)u\ra - C\la \phi_1(V), u\ra^2, \quad \forall u \in H_0^1(\Omega)\cap H^2(\Omega).
\label{eq:back-heat-H2coer-2}
\end{gather}
\end{lemma}
\begin{proof}
By H\"older and Sobolev, we have
\begin{equation*}
\int_\Omega V\phi_1(V)^2\ud x \leq \|V\|_{L^p}\|\phi_1(V)\|_{L^2}^{2-\frac dp}\|\phi_1(V)\|_{L^{\frac{2d}{d-2}}}^\frac dp
\leq C(\Omega) M \|\phi_1(V)\|_{H_0^1}^\frac dp.
\end{equation*}
Thus
\begin{equation*}
0 \leq -\lambda_1(V) = \int_\Omega \big(V \phi_1(V)^2  - |\grad \phi_1(V)|^2 \big)\ud x \leq C(\Omega) M\|\phi_1(V)\|_{H_0^1}^\frac dp
- \|\phi_1(V)\|_{H_0^1}^2,
\end{equation*}
which implies
\begin{equation*}
\|\phi_1(V)\|_{H_0^1} \leq (C(\Omega) M)^\frac{p}{2p-d}, \qquad |\lambda_1(V)| \leq (C(\Omega)M)^\frac{2p}{2p-d}.
\end{equation*}

In order to prove \eqref{eq:lambda12-bound}, we observe that
\begin{equation*}
\begin{aligned}
\lambda_1(W) &\leq \int_\Omega \big(|\grad \phi_1(V)|^2 + W|\phi_1(V)|^2\big)\ud x \\ &\leq \lambda_1(V) + \int_\Omega |V - W||\phi_1(V)|^2 \ud x \leq \lambda_1(V) + C\|V - W\|_{L^p},
\end{aligned}
\end{equation*}
where the last step follows from H\"older, Sobolev and \eqref{eq:lambda1-bound}. Analogously, $\lambda_1(V) \leq \lambda_1(W) + C\|V - W\|_{L^p}$.

Next, we prove \eqref{eq:back-heat-coer}. By the Spectral Theorem we have
\begin{equation}
\label{eq:spectral-coer}
\int_\Omega \big(|\grad u|^2 + Vu^2\big)\ud x + (\mu - \lambda_1(V))\la \phi_1(V), u\ra^2 \geq \mu \|u\|_{L^2}^2.
\end{equation}
For any $\eta > 0$ we thus have
\begin{equation*}
\begin{aligned}
&\int_\Omega \big(|\grad u|^2 + Vu^2\big)\ud x + (1-\eta)(\mu - \lambda_1(V))\la \phi_1(V), u\ra^2 
\geq (1-\eta)\mu\|u\|_{L^2}^2 + \eta \int_\Omega \big(|\grad u|^2 + Vu^2\big)\ud x \\
&\geq \eta \|u\|_{H_0^1}^2 - \eta C\|u\|_{L^2}^{2-\frac dp}\|u\|_{H_0^1}^\frac dp + (1-\eta)\mu\|u\|_{L^2}^2,
\end{aligned}
\end{equation*}
so if we take $\eta$ small enough, then the Young's inequality for products yields \eqref{eq:back-heat-coer}.

We prove \eqref{eq:phi12-bound}. Using the bounds already proved, we have
\begin{equation*}
\int_\Omega \big(|\grad \phi_1(W)|^2 + V \phi_1(W)^2\big)\ud x \leq \lambda_1(W) + C\|V - W\|_{L^p} \leq \lambda_1(V) + C\|V - W\|_{L^p}.
\end{equation*}
Let $\phi_1(W) = a\phi_1(V) + bu$, with $a^2 + b^2 = 1$, $\|u\|_{L^2} = 1$ and $\la \phi_1(V), u\ra = 0$. We then have
\begin{equation*}
\lambda_1(V) a^2 + \mu b^2 \leq\int_\Omega \big(|\grad \phi_1(W)|^2 + V \phi_1(W)^2\big)\ud x \leq \lambda_1(V) + C\|V - W\|_{L^p},
\end{equation*}
thus $|b| \leq C\sqrt{\|V - W\|_{L^p}}$, which implies
\begin{equation*}
\|\phi_1(W) \pm \phi_1(V)\|_{L^2} \leq C\sqrt{\|V - W\|_{L^p}}\quad \Rightarrow \quad \|\phi_1(W) - \phi_1(V)\|_{L^2} \leq C\sqrt{\|V - W\|_{L^p}},
\end{equation*}
where the last implication follows because both functions are positive. Now \eqref{eq:phi12-bound} easily follows from \eqref{eq:back-heat-coer} for $u := \phi_1(V) - \phi_1(W)$. Indeed, we have
\begin{equation*}
\begin{aligned}
&\la \phi_1(V) - \phi_1(W), ({-}\Delta + V)(\phi_1(V) - \phi_1(W))\ra \\
&\leq \la \phi_1(V) - \phi_1(W), \lambda_1(V)\phi_1(V) - \lambda_1(W)\phi_1(W)\ra + C\|V - W\|_{L^p} \leq C\|V - W\|_{L^p}.
\end{aligned}
\end{equation*}

Inequality \eqref{eq:back-heat-H2coer-2} follows from \eqref{eq:spectral-coer} applied to $\sqrt{L}u$ instead of $u$,
where $Lu := ({-}\Delta + V)u + (\mu - \lambda_1(V))\la \phi_1(V), u\ra\phi_1(V)$.

Finally, in order to prove \eqref{eq:back-heat-H2coer}, we write
\begin{equation*}
\begin{aligned}
\|({-}\Delta + V)u\|_{L^2}^2 &= \eta\|({-}\Delta + V)u\|_{L^2}^2 + (1-\eta)\|({-}\Delta + V)u\|_{L^2}^2 \\
&\geq \frac{\eta}{2}\|{-}\Delta u\|_{L^2}^2 - 2\eta \|Vu\|_{L^2}^2  + (1-\eta)\|({-}\Delta + V)u\|_{L^2}^2.
\end{aligned}
\end{equation*}
By the Sobolev inequality, we have $\|Vu\|_{L^2} \leq C\|{-}\Delta u\|_{L^2}^{1-\alpha }\|u\|_{L^2}^\alpha $ for some $\alpha > 0$.
Thus, if we take $\eta$ small enough, \eqref{eq:back-heat-H2coer} follows from \eqref{eq:back-heat-H2coer-2} and \eqref{eq:back-heat-coer}.
\end{proof}
\begin{proposition}
\label{prop:back-heat-cauchy}
If $r$ is large enough, then for any $V \in L_\tx{loc}^r([0, \infty), L^p)$ equation \eqref{eq:back-heat} defines
a strongly continuous backward evolution operator $S(\tau, t)$ in $H_0^1(\Omega)$.
Moreover, for any $[\tau, t] \subset [0, \infty)$ the mapping
\begin{equation*}
L^r([\tau, t], L^p) \owns V \mapsto S(\tau, t) \in \scrL(H_0^1(\Omega))
\end{equation*}
is continuous.
\end{proposition}
\begin{proof}
Since $\Omega$ is bounded, without loss of generality we can assume $p < d$. 
Set $q := \big(\frac 1p + \frac{d-2}{2d}\big)^{-1} \in \big(\frac{2d}{d+2}, 2\big)$,
so that $\|V u\|_{L^q} \lesssim \|V\|_{L^p}\|u\|_{H_0^1}$.
Using the regularising effect of the heat-flow and the $L^p$--$L^q$ estimates, see \cite[pages 42--44]{CaHa}, we get for all $t > 0$
\begin{equation}
\label{eq:heat-smoothing}
\|\eee^{t\Delta} u\|_{H_0^1} = \|\eee^{\frac t2 \Delta}\eee^{\frac t2 \Delta}u\|_{H_0^1} \lesssim t^{-\frac 12}\|\eee^{\frac t2 \Delta}u\|_{L^2} \lesssim t^{-\frac 12 -\frac d2\big(\frac 1q - \frac 12\big)}\|u\|_{L^q} \lesssim t^{-\beta}\|u\|_{L^q},
\end{equation}
where $\beta := \frac 12 + \frac d2\big(\frac 1q - \frac 12\big) \in \big(\frac 12, 1\big)$ and the constant depends only on $\Omega$ and $p$.

Fix $\tau \leq t$, denote $I := [\tau, t]$ and consider the bilinear operator
\begin{equation*}
\Phi: L^r(I, L^p)\times C(I, H_0^1) \to C(I, H_0^1),\qquad \Phi(V, w)(s) := \int_s^t \eee^{(s'-s)\Delta}(V(s')w(s'))\ud s'.
\end{equation*}
Take $r := 2(1-\beta)^{-1}$ (in fact any $r \in ((1-\beta)^{-1}, \infty)$ would work). From \eqref{eq:heat-smoothing} we obtain
\begin{equation*}
\|\Phi(V, w)\|_{L^\infty H_0^1} \lesssim (t - \tau)^\frac{1-\beta}{2}\|V\|_{L^r L^p}\|w\|_{L^\infty H_0^1},
\end{equation*}
thus there exists $c_0 > 0$ such that if $t - \tau \leq c_0 \|V\|_{L^r L^p}^\frac{-2}{1-\beta}$,
then $\|\Phi(V, \cdot)\|_{\scrL(C(I, H_0^1))} \leq \frac 12$.
One can check that if we define
\begin{equation*}
S(s, t) := \big((\Id + \Phi(V, \cdot))^{-1}(\eee^{(t - \cdot)\Delta})\big)(s),\qquad \forall s \in [\tau, t],
\end{equation*}
then
\begin{equation*}
S(s, t) = \eee^{(t-s)\Delta}u - \int_s^t \eee^{(s'-s)\Delta}(V(s')S(s', t)u)\ud s',
\end{equation*}
which means that $S(s, t)u$ satisfies the integral form of \eqref{eq:back-heat}.
We see that $S(s, t)$ depends continuously on $V$.

This finishes the proof for sufficiently short time intervals. In general, we divide any given time interval into a finite number
of sufficiently short subintervals.
\end{proof}
\begin{proof}[Proof of Proposition~\ref{prop:back-heat}]
We will obtain the result as a corollary of Theorem~\ref{thm:main-cont-1}.
We assume $V \in L^\infty([0, \infty), L^p)$, in particular $V \in L_\tx{loc}^r([0, \infty), L^p)$ for any $r$,
thus Proposition~\ref{prop:back-heat-cauchy} implies that \eqref{eq:back-heat} defines a strongly continuous backward evolution operator.

Let $\chi$ be a $C^\infty$ positive function supported in $(-\frac 12, \frac 12)$ such that $\int_\bR \chi(x) \ud x = 1$.
We set
\begin{equation*}
W(t) := \int_\bR \chi(t - \tau)V(\tau)\ud\tau, \qquad W \in C^\infty((1/2, \infty), L^p).
\end{equation*}
Observe that $\|W'(t)\|_{L^p} \lesssim \delta$ and
\begin{equation*}
\label{eq:V-W-close}
\|V(t) - W(t)\|_{L^p} \leq \int_\bR \chi(t-\tau)\|V(t) - V(\tau)\|_{L^p}\ud\tau \leq \delta.
\end{equation*}
By Lemma~\ref{lem:back-heat-spectral}, we have $\lambda_1(W(t)) \leq -\mu + \frac{1}{10}\epsilon$
and $\lambda_2(W(t)) \geq \mu - \frac{1}{10}\epsilon$ for all $t$ if $\delta$ is small enough.
We set $\phi(t) := \phi_1(W(t))$ and
\begin{align*}
I_t^-(v) &:= C_0|\la \phi_1(W(t)), v\ra|, \\
I_t^+(v) &:= \sqrt{\max\Big(0, \int_\Omega\big(|\grad v|^2 + W(t)v^2\big)\ud x\Big)},
\end{align*}
where $C = C(\Omega, p, M)$ is a large constant.
Clearly, $I_t^-(v)$ and $I_t^+(v)$ are continuous with respect to $(t, v)$
(for $I_t^-$ we use Lemma~\ref{lem:back-heat-spectral}).

Assumption \eqref{eq:norm-equiv-c} follows from \eqref{eq:back-heat-coer}, if $C_0$ is large enough.
Assumption \eqref{eq:Inm-form-c} obviously holds.
It is also clear that $I_t^+(\phi_1(W(t))) = 0$, which implies \eqref{eq:contains-K-c} for any choice of $c_4$.
We now prove that assumption \eqref{eq:ineq-Im-c} holds.
We will choose $c_4$ later (we will see that $c_4$ can be chosen as large as we want,
in particular we can guarantee that \eqref{eq:intro-c3c4-c} holds).

Let $u$ be a solution of \eqref{eq:back-heat} and let $\tau \leq t$ be such that $I_s^-(u(s)) \geq \frac{1}{c_4}I_s^+(u(s))$
for all $s \in [\tau, t]$,
which implies $\|u(s)\|_{H_0^1} \lesssim I_s^-(u(s))$, see \eqref{eq:norm-equiv-Im}.
Suppose that \eqref{eq:ineq-Im-c} fails and set
\begin{equation*}
t_0 := \inf\{t': I_s^-(u(s)) \geq \eee^{({-}\mu + \epsilon)(s - t)}I_t^-(u(t))\text{ for all }s \in [t', t]\}
\end{equation*}
(we allow the possibility $t_0 = t$). By continuity, $I_{t_0}^-(u(t_0)) \geq \eee^{({-}\mu + \epsilon)(t_0 - t)}I_t^-(u(t))$.
To reach a contradiction, it suffices to show
\begin{equation}
\label{eq:back-heat-boot}
I_{s}^-(u(s)) \geq (1 + ({-}\mu + \epsilon/2)(s - t_0))I_{t_0}^-(u(t_0)) \qquad \text{for }s \in [t_0 - \eta, t_0]\text{ for some }\eta > 0.
\end{equation}

Set $\phi := \phi(t_0)$. If $\eta > 0$ is small, then for $s \in [t_0 - \eta, t_0]$ we have
\begin{align*}
\|u(s) - u(t_0)\|_{H_0^1} &\ll \|u(t_0)\|_{H_0^1}, \\
\|W(t_0) - V(s)\|_{L^p} \lesssim \delta &\ll 1, \\
\|\phi - \phi(s)\|_{L^2} \lesssim \delta|s - t_0| &\ll |s - t_0|.
\end{align*}
Below, we write ``$\simeq$'' when we mean ``up to terms $\ll |s - t_0|\|u(t_0)\|$''.
\begin{equation*}
\begin{aligned}
\la \phi(s), u(s)\ra - \la \phi, u(t_0)\ra &\simeq \la \phi, u(s)\ra - \la \phi, u(t_0)\ra \\
& = \int_s^{t_0} \la (\Delta - V(s'))\phi, u(s')\ra\ud s' \\
&= \int_s^{t_0} \la (\Delta - W(t_0))\phi, u(s')\ra\ud s' + \int_s^{t_0} \la (W(t_0) - V(s'))\phi, u(s')\ra \ud s' \\
&\simeq -\lambda_1(W(t_0))\int_s^{t_0}\la \phi, u(s')\ra \simeq \lambda_1(W(t_0))(s - t_0)\la \phi, u(t_0)\ra.
\end{aligned}
\end{equation*}
Since $\lambda_1(W(t_0)) \leq -\mu + \frac{1}{10}\epsilon$ and $s - t_0 \leq 0$, we obtain \eqref{eq:back-heat-boot}.

We proceed similarly with $I_t^+$.
Let $\tau \leq t$ be such that $I_s^+(u(s)) \geq c_3 I_s^-(u(s))$ for all $s \in [\tau, t]$,
which implies $\|u(s)\|_{H_0^1} \lesssim I_s^+(u(s))$, see \eqref{eq:norm-equiv-Ip}.
If $I_s^+(u(s)) = 0$ for some $s \in [\tau, t]$, then the solution is identically $0$, so assume $I_s^+(u(s)) > 0$ for all $s \in [\tau, t]$.
Suppose that \eqref{eq:ineq-Ip-c} fails and set
\begin{equation*}
t_0 := \inf\{t': I_s^+(u(s)) \leq \eee^{(\mu - \epsilon)(s - t)}I_t^+(u(t))\text{ for all }s \in [t', t]\}
\end{equation*}
(we allow the possibility $t_0 = t$). By continuity, $I_{t_0}^+(u(t_0)) \geq \eee^{(\mu - \epsilon)(t_0 - t)}I_t^-(u(t))$.
To reach a contradiction, it suffices to show
\begin{equation}
\label{eq:back-heat-boot-2}
I_{s}^+(u(s)) \leq (1 + (\mu - \epsilon/2)(s - t_0))I_{t_0}^+(u(t_0)) \qquad \text{for }s \in [t_0 - \eta, t_0]\text{ for some }\eta > 0.
\end{equation}
If $\eta > 0$ is small, then for all $s \in [t_0 - \eta, t_0]$ we have
\begin{equation*}
\|u(s) - u(t_0)\|_{H_0^1} \ll \|u(t_0)\|_{H_0^1}, \quad  \|\partial_s W(s)\|_{L^p} \lesssim \delta \ll 1.
\end{equation*}
Approximating $V$ by a smooth potential in the norm $L^rL^p$ and using Proposition~\ref{prop:back-heat-cauchy},
we can assume that $u$ and $V$ are smooth in both space and time. In the computation below, ``$\simeq$'' means
``up to terms $\ll \|u(t_0)\|_{H_0^1}$''.
\begin{equation*}
\begin{aligned}
\frac 12\dd s I_s^+(u(s))^2 &= \la \partial_s W(s) u(s), u(s)\ra
+ \la ({-}\Delta + W(s))u(s), ({-}\Delta + V(s))u(s)\ra \\
&\simeq \|({-}\Delta + W(s))u(s)\|_{L^2}^2 + \la ({-}\Delta + W(s))u(s), (V(s)-W(s))u(s)\ra  \\
&\geq (\mu - \epsilon/8)I_s^+(u(s))^2 + \frac{\epsilon}{C}\|{-}\Delta u(s)\|_{L^2}^2 - \|V - W\|_{L^p}\|{-}\Delta u(s)\|_{L^2}^2
- C I_s^-(u(s))^2.
\end{aligned}
\end{equation*}
Thus, if $c_3 = c_3(\epsilon)$ is large enough and $I_s^+(u(s)) \geq c_3 I_s^-(u(s))$, then
\begin{equation*}
\frac 12 \dd s I_s^+(u(s))^2 \geq (\nu - \epsilon / 4)I_s^+(u(s))^2,
\end{equation*}
which implies \eqref{eq:back-heat-boot-2}.
\end{proof}
\begin{remark}
For ordinary differential equations, a more general result (dealing with the non self-adjoint case) is proved in \cite[Chapter 6]{Coppel}.
\end{remark}
\begin{remark}
We expect that a similar result could be obtained in $H^k \cap H_0^1$ for any $k \in \{0, 1, 2, \ldots\}$.
We would use the functional $I_t^+(v) := \sum_{j=0}^k a_j \la v, ({-}\Delta + W(t))^j v\ra$,
for appropriate strictly positive numbers $a_0, \ldots, a_k$.
\end{remark}
\subsection{Heat equation with slowly moving potentials}
\label{ssec:heat-forward}
Let $d \geq 3$ and let $V_1, \ldots V_J \in L^\infty(\bR^d) \cap L^\frac d2(\bR^d)$ be potentials.
Assume $-\Delta + V_j$ has $K_j$ strictly negative eigenvalues $-\lambda_{j, 1}, \ldots, -\lambda_{j, k_j}$,
with corresponding eigenfunctions $\cY_{j, 1}, \ldots, \cY_{j, k_j}$.
Note that $K_j$ is finite by the Cwikel-Lieb-Rozenblum theorem.
Let $K := \sum_{j=1}^J K_j$ and $\lambda := \min\{\lambda_{j, k}: 1\leq j\leq J,\ 1\leq k\leq k_j\}$.
For $j \in \{1, \ldots, J\}$, let $x_j: [0, \infty) \to \bR^d$ be a $C^1$ trajectory.

We consider the heat equation with moving potentials:
\begin{equation*}
\label{eq:heat}
\partial_t u(t, x) = \Delta u(t, x) - \sum_{j=1}^J V_j(x - x_j(t))u(t, x).
\end{equation*}
By standard arguments, similar to the one given in Section~\ref{ssec:backward-heat},
this defines a strongly continuous dynamical system $T(\tau, t): L^2(\bR^d) \to L^2(\bR^d)$ for $\tau \geq t$.
\begin{proposition}
\label{prop:heat}
For any $\epsilon > 0$ there exists $\eta > 0$ such that if for $t$ large enough
$|x_j'(t)| \leq \eta$ and $|x_j(t) - x_l(t)| \geq \frac{1}{\eta}$ for all $j \neq l$,
then $T(\tau, t)$ has an exponential dichotomy with exponents $\epsilon$ and $\lambda - \epsilon$.
Moreover, $\codim X_\tx s = K$.
\end{proposition}
We need the following fact.
\begin{lemma}
\label{lem:coer-1}
Let $V \in L^\frac d2(\bR^d)$ and let ${-}\lambda_1, \ldots, -\lambda_{K_0}$
be the strictly negative eigenvalues of $-\Delta + V$,
with corresponding eigenfunctions $\cY_1, \ldots, \cY_{K_0}$.
Then there exists $C_0 > 0$ such that for all $u \in L^2(\bR^d)$ the following inequality is true:
\begin{equation}
\label{eq:coer-1}
\int_{\bR^d} \big(|\grad u|^2 + V u^2 \big)\ud x + C_0\sum_{k=1}^K \la \cY_k, u\ra^2 \geq 0.
\end{equation}
\end{lemma}
\begin{proof}
The self-adjoint operator corresponding to the quadratic form in \eqref{eq:coer-1} is
\begin{equation*}
\label{eq:coer-1-op}
L := -\Delta + V + C_0 \sum_{k=1}^K \la \cY_k, \cdot\ra \cY_k.
\end{equation*}
The spaces $Y := \spn(\cY_1, \ldots, \cY_K)$ and its orthogonal complement (in $L^2$) $Y^\perp$ are invariant subspaces of $L$.
On $Y^\perp$, the quadratic form is positive by the Spectral Theorem. On $Y$, it is positive if we take $C_0$ large enough.
\end{proof}
%
This easily implies coercivity for multiple potentials.
\begin{lemma}
\label{lem:coer-2}
Let $V_1, \ldots V_J \in L^\frac d2(\bR^d)$ be potentials.
Assume $-\Delta + V_j$ has $K_j$ strictly negative eigenvalues $-\lambda_{j, 1}, \ldots, -\lambda_{j, k_j}$,
with corresponding eigenfunctions $\cY_{j, 1}, \ldots, \cY_{j, k_j}$.
There exists $C_0 > 0$ with the following property.
For any $\epsilon > 0$ there exists $\eta > 0$ such that if $|x_j - x_l| \geq \frac{1}{\eta}$ for all $j \neq l$,
then for all $u \in L^2(\bR^d)$ the following bound holds:
\begin{equation*}
\label{eq:coer-2}
\int_{\bR^d} \Big(|\grad u|^2 + \sum_{j=1}^J V_j(\cdot - x_j) u^2 \Big)\ud x
+ C_0\sum_{j=1}^J\sum_{k=1}^{K_j} \la \cY_{j,k}(\cdot - x_j), u\ra^2 \geq -\epsilon \|u\|_{L^2}^2.
\end{equation*}
\end{lemma}
\begin{proof}
Let $\chi$ be a smooth cut-off function, with the support contained in $B(0, \frac 12)$ and equal $1$ on $B(0, \frac 14)$.
For $j \in \{1, \ldots, J\}$, let $u_j(x) := \chi(\eta(x - x_j))u(x)$.
We obtain the result by summing \eqref{eq:coer-1}, applied for $u_j$ instead of $u$ for $j \in \{1, \ldots, J\}$.
\end{proof}
\begin{proof}[Proof of Proposition~\ref{prop:heat}]
For $t \geq t_0$ and $v \in L^2(\bR^d)$, we define
\begin{equation*}
I_t^+(u) := \sum_{j=1}^J \sum_{k=1}^{K_j} |\la \cY_{j, k}(\cdot - x_j(t)), u\ra|
\end{equation*}
and
\begin{equation*}
I_t^-(u) := \|u\|_{L^2}.
\end{equation*}
We have
\begin{equation*}
\dd t\|u(t)\|_{L^2}^2 = {-}\int_{\bR^d} \Big(|\grad u(t)|^2 + \sum_{j=1}^J V_j(\cdot - x_{j}) u(t)^2 \Big)\ud x,
\end{equation*}
so the required sub-exponential growth of $I^-$ \eqref{eq:ineq-Im-c-2} follows from Lemma~\ref{lem:coer-2},
provided that $c_3$ is taken small enough.

In order to check \eqref{eq:ineq-Ip-c-2}, we compute
\begin{equation*}
\begin{aligned}
&\dd t\la \cY_{j_0, k_0}(\cdot - x_{j_0}(t)), u(t)\ra = \la \cY_{j_0, k_0}(\cdot - x_{j_0}(t)), \Delta u(t) - V_{j_0}(\cdot - x_{j_0}(t))u(t)\ra \\
&-x_{j_0}'(t)\cdot \la \grad \cY_{j_0, k_0}(\cdot - x_{j_0}(t)), u(t)\ra - \sum_{j\neq j_0}\la \cY_{j_0, k_0}(\cdot - x_{j_0}(t)),
V_j(\cdot - x_j(t))u(t)\ra.
\end{aligned}
\end{equation*}
We see that the second line is negligible when $\eta$ is small, more precisely for any $\wt\varepsilon > 0$
\begin{equation*}
\Big|\dd t\la \cY_{j_0, k_0}(\cdot - x_{j_0}(t)), u(t)\ra - \lambda_{j_0, k_0}\la \cY_{j_0, k_0}(\cdot - x_{j_0}(t)), u(t)\ra \Big|
\leq \wt\varepsilon \|u(t)\|_{L^2}.
\end{equation*}
if $\eta$ is small enough. From this we deduce
\begin{equation*}
\Big|\dd t|\la \cY_{j_0, k_0}(\cdot - x_{j_0}(t)), u(t)\ra| - \lambda_{j_0, k_0}|\la \cY_{j_0, k_0}(\cdot - x_{j_0}(t)), u(t)\ra| \Big|
\leq \wt\varepsilon \|u(t)\|_{L^2},
\end{equation*}
which in turn yields
\begin{equation*}
\dd t I^+_t(u(t)) \geq \lambda I^+_t(u(t)) - K\wt \varepsilon\|u(t)\|_{L^2}.
\end{equation*}
It suffices to take $c_4$ some number satisfying \eqref{eq:intro-c3c4-c-2} and $\wt \varepsilon = \frac{c_4\varepsilon}{K}$.
\end{proof}

\section{Klein-Gordon equation with potentials having almost constant velocity}
\label{sec:wave}

\subsection{One potential}
\label{ssec:kg-one}
The purpose of this section is to relate properties of the flow with one moving
potential to the properties of the corresponding flow with a stationary potential.
We will use many facts from \cite{CoMa-multi}.

We will need the Lorentz boosts. Let $\beta \in \bR^d$, $|\beta| < 1$, be a velocity vector.
For a function $\phi : \bR^d \to \bR$ we define
\begin{equation*}
\label{eq:lorentz-boost}
\phi_{\beta}(x) := \phi(\Lambda_{\beta}x),\quad \Lambda_{\beta}x := x + (\gamma-1)\frac{(\beta\cdot x)\beta}{|\beta|^2},\quad \gamma := \frac{1}{\sqrt{1-|\beta|^2}}.
\end{equation*}
With this notation, the Lorentz transformation is given by
\begin{equation*}
\label{eq:lorentz}
(t', x') = \big(\gamma(t - \beta \cdot x),\ \Lambda_\beta x - \gamma\beta t\big)
= \big(\gamma(t - \beta\cdot x),\ \Lambda_\beta(x - \beta t)\big).
\end{equation*}

For a pair of functions $\bR^d \to \bR$,  $\bs \phi = (\phi, \dot\phi)$, we will also write
\begin{equation*}
\bs \phi_\beta := (\phi_\beta, \dot \phi_\beta).
\end{equation*}

Let $V$ be a smooth exponentially decaying potential.
Let $\beta\in \bR^d$ with $|\beta| < 1$
and $\xi \in \bR^d$. We consider the following linear
Klein-Gordon equation:
\begin{equation}
\label{eq:kg-one-2ord}
\partial_t^2 u(t, x) = \Delta u(t, x) - u(t, x) - V_\beta(x - \beta t - \xi)u(t, x),\quad (t, x) \in \bR \times \bR^d.
\end{equation}
We see that $u(t, x)$ is a solution of \eqref{eq:kg-one-2ord} if and only if $u(t, x) = w(t', x')$,
where
\begin{equation}
\label{eq:t'x'def}
(t', x') = (\gamma(t - \beta\cdot x),\ \Lambda_\beta(x - \beta t - \xi))
\end{equation}
and $w(t, x)$ is a solution of
\begin{equation}
\label{eq:kg-one-2ord-tindep}
\partial_t^2 w = \Delta w - w - V w.
\end{equation}
The following observation will be useful. It expresses the conservation of charge
(for a complex-valued solution $w = w_1 + iw_2)$ and energy
for equation \eqref{eq:kg-one-2ord-tindep}, which is a consequence of Noether's Theorem.
\begin{lemma}
\label{lem:kg-div-free}
Let $w, w_1, w_2$ be smooth solutions of \eqref{eq:kg-one-2ord-tindep}. The following vector fields in $\bR^{1+d}$
are divergence free:
\begin{align*}
F(w_1, w_2) &= (w_1\partial_t w_2 - w_2 \partial_t w_1, -w_1\grad w_2 + w_2\grad w_1), \\
G(w) &= \Big(\frac 12\big((\partial_t w)^2 + (\grad w)^2 + (1+V)w^2\big), -\partial_t w\grad w\Big).
\end{align*}
\end{lemma}
\begin{proof}
We have
\begin{equation*}
\begin{aligned}
&\partial_t (w_1\partial_t w_2 - w_2\partial_t w_1) + \div({-}w_1\grad w_2 + w_2\grad w_1)
= w_1\partial_t^2 w_2 - w_2\partial_t^2 w_1 - w_1\Delta w_2 + w_2 \Delta w_1 \\
&= w_1 (\partial_t^2 w_2 - \Delta w_2) - w_2(\partial_t^2 w_1 - \Delta w_1) = -w_1(w_2 + Vw_2) + w_2(w_1 + Vw_1) = 0
\end{aligned}
\end{equation*}
and
\begin{equation*}
\begin{aligned}
&\frac 12\partial_t \big((\partial_t w)^2 + (\grad w)^2 + (1+V)w^2\big)
+ \div({-}\partial_t w\grad w) \\
&= \partial_t^2 w \partial_t w + \partial_t\grad w\cdot \grad w+ (1+V)w\partial_t w - \partial_t \grad w\cdot \grad w - \partial_t w \Delta w  \\
&= \partial_t w(\partial_t^2 w - \Delta w + (1+V) w) = 0.
\end{aligned}
\end{equation*}
\end{proof}
We can write \eqref{eq:kg-one-2ord} as a dynamical system:
\begin{equation}
\label{eq:kg-one-ham}
\partial_t \bs u(t) = JH_\beta(\beta t + \xi) \bs u(t),
\end{equation}
where
\begin{equation}
\label{eq:Hbeta-def}
J := \begin{pmatrix} 0 & 1 \\ -1 & 0 \end{pmatrix},\qquad
H_\beta(\xi) := \begin{pmatrix} -\Delta +1 + V_\beta(\cdot - \xi) & 0 \\ 0 & 1 \end{pmatrix}.
\end{equation}

The Schr\"odinger operator $L = -\Delta + 1 + V$ has the essential spectrum $[1, \infty)$,
and might have a finite number of eigenvalues in $(-\infty, \frac 12)$.
Let $K$ be the number of strictly negative eigenvalues (counted with multiplicities) and let $M := \dim\ker L$.
Let $-\nu_1^2, \ldots -\nu_K^2$ (with $\nu_k > 0$) be the strictly negative eigenvalues
and let $(\phi_k)_{k = 1, \ldots, K}$ and $(\phi_m^0)_{j = 1, \ldots, M}$
be orthonormal (in $L^2$) families such that
\begin{align*}
L \phi_k &= -\nu_k^2 \phi_k, \\
L \phi_m^0 &= 0.
\end{align*}
\begin{lemma}
\label{lem:eigen-decay}
For any $\varepsilon > 0$ and $n \in \bN^d$ there exists $C > 0$ such that
\begin{align*}
|\partial^n\phi_k(x)| &\leq C\eee^{-(1-\varepsilon)\sqrt{1+\nu_k^2}|x|}, \quad \text{for all }k \in \{1, \ldots, K\},\ x \in \bR^d, \\
|\partial^n\phi_m^0(x)| &\leq C\eee^{-(1-\varepsilon)|x|}, \quad \text{for all }m \in \{1, \ldots, M\},\ x \in \bR^d. 
\end{align*}
\end{lemma}
\begin{proof}
We only prove the second inequality, as the first one can be obtained in the same way.
Take $C$ large and suppose there exists $x \in \bR$ such that $\phi_m^0(x) > C\eee^{-(1-\varepsilon)|x|}$.
By interior regularity we have $\lim_{|x| \to \infty}|\phi_m^0(x)| = 0$, so there exists $x_0 \in \bR^d$ such that
\begin{equation*}
\phi_m^0(x_0) - C\eee^{-(1-\varepsilon)|x_0|} = \sup_{x \in \bR^d}\big(\phi_m^0(x) - C\eee^{-(1-\varepsilon)|x|}\big) > 0,
\end{equation*}
which implies
\begin{equation*}
\Delta (\phi_m^0(x_0)) \leq C \Delta (\eee^{-(1-\varepsilon)|x_0|}).
\end{equation*}
If $C$ is large, then $|x_0|$ is large as well. It is easy to see from the formula for the Laplacian in radial coordinates
that if $|x_0|$ is large enough, then $\Delta (\eee^{-(1-\varepsilon)|x_0|}) \leq (1-\epsilon)\eee^{-(1-\varepsilon)|x_0|}$.
We obtain
\begin{equation*}
\Delta (\phi_m^0(x_0)) \leq C(1-\epsilon)\eee^{-(1-\varepsilon)|x_0|} < (1-\epsilon) \phi_m^0(x_0),
\end{equation*}
which is impossible for $|x_0|$ large.

The bound $\phi_m^0(x) > -C\eee^{-(1-\varepsilon)|x_0|}$ is obtained by considering $-\phi_m^0$ instead of $\phi_m^0$.

The bound on derivatives follows from interior regularity.
\end{proof}
By the Spectral Theorem,
\begin{equation}
\label{eq:L2-coer}
\la \phi_k, \psi\ra = \la \phi_m^0, \psi\ra = 0 \quad\Rightarrow\quad \la \psi, L\psi\ra \geq c\|\psi\|_{L^2}^2.
\end{equation}
Note that we also have
\begin{equation}
\label{eq:H1-coer}
\la \phi_k, \psi\ra = \la \phi_m^0, \psi\ra = 0 \quad\Rightarrow\quad \la \psi, L\psi\ra \geq c\|\psi\|_{H^1}^2.
\end{equation}
Indeed, \eqref{eq:L2-coer} implies that $\la \psi, (a(-\Delta + 1) + V)\psi\ra \geq 0$
for some $a < 1$, which yields \eqref{eq:H1-coer} with $c = 1-a$.

Following \cite{CoMu2014} and \cite[Lemma 1]{CoMa-multi}, we now give explicit formulas for the stable, unstable
and null components of the flow \eqref{eq:kg-one-ham}. We define
\begin{align}
\cY_{k, \beta}^-(x) &:= \eee^{\gamma\nu_k\beta\cdot x}(\phi_k, \gamma\beta\cdot\grad \phi_k+\gamma\nu_k\phi_k)_\beta(x), \label{eq:kg-Ym-def} \\
\cY_{k, \beta}^+(x) &:= \eee^{-\gamma\nu_k\beta\cdot x}
(\phi_k, -\gamma\beta\cdot\grad \phi_k+\gamma\nu_k\phi_k)_\beta(x),\label{eq:kg-Yp-def} 
\\ \cY_{m, \beta}^0(x) &:= (\phi_m^0, -\gamma\beta\cdot\grad \phi_m^0)_\beta(x),\label{eq:kg-Y0-def} 
\\ \alpha_{k, \beta}^-(x) &:= J\cY_{k, \beta}^+(x) = \eee^{-\gamma\nu_k\beta\cdot x}(\gamma\beta\cdot\grad \phi_k-\gamma\nu_k\phi_k, \phi_k)_\beta(x), \label{eq:kg-am-def} \\
\alpha_{k, \beta}^+(x) &:= J\cY_{k, \beta}^-(x) = \eee^{\gamma\nu_k\beta\cdot x}({-}\gamma\beta\cdot\grad \phi_k-\gamma\nu_k\phi_k, \phi_k)_\beta(x), \label{eq:kg-ap-def} \\
\alpha_{m, \beta}^0(x) &:= J\cY_{m, \beta}^0(x) = (\gamma \beta\cdot\grad \phi_m^0, \phi_m^0)_\beta(x).\label{eq:kg-a0-def} 
\end{align}
Since $|\Lambda_\beta x| \geq \gamma|x|$, Lemma~\ref{lem:eigen-decay} implies that all these functions
are smooth and exponentially decaying, uniformly in $\beta$ of $|\beta| \leq v < 1$. Observe also that
\begin{equation}
\label{eq:beta-derivative}
\|\partial_\beta \alpha_{k, \beta}^-\|_{L^2} +\|\partial_\beta \alpha_{k, \beta}^+\|_{L^2}+\|\partial_\beta \alpha_{m, \beta}^0\|_{L^2}
+\|\partial_\beta \cY_{k, \beta}^-\|_{L^2}+\|\partial_\beta \cY_{k, \beta}^+\|_{L^2}+\|\partial_\beta \cY_{m, \beta}^0\|_{L^2} \lesssim 1.
\end{equation}
\begin{lemma}
\label{lem:kg-one-comp}
The following functions are solutions of \eqref{eq:kg-one-ham}:
\begin{align}
\bs u(t) &= \eee^{-\frac{\nu_k}{\gamma} t}\cY_{k, \beta}^-(\cdot - \beta t - \xi), \label{eq:part-sol-1} \\
\bs u(t) &= \eee^{\frac{\nu_k}{\gamma} t}\cY_{k, \beta}^+(\cdot - \beta t - \xi), \label{eq:part-sol-2} \\
\bs u(t) &= \cY_{m, \beta}^0(\cdot - \beta t - \xi).\label{eq:part-sol-3} 
\end{align}
If $\bs u(t)$ is any solution of \eqref{eq:kg-one-ham}, then
\begin{align}
\dd t\la \alpha_{k, \beta}^-(\cdot - \beta t - \xi), \bs u(t)\ra &= -\frac{\nu_k}{\gamma} \la \alpha_k^-(\cdot - \beta t - \xi), \bs u(t)\ra, \label{eq:covector-1} \\
\dd t\la \alpha_{k, \beta}^+(\cdot - \beta t - \xi), \bs u(t)\ra &= \frac{\nu_k}{\gamma} \la \alpha_k^+(\cdot - \beta t - \xi), \bs u(t)\ra, \label{eq:covector-2} \\
\dd t\la \alpha_{m, \beta}^0(\cdot - \beta t - \xi), \bs u(t)\ra &= 0.\label{eq:covector-3} 
\end{align}
\end{lemma}
\begin{proof}
It is easy to see that $w(t, x) = \eee^{\nu_k t}\phi_k(x)$ is a solution of \eqref{eq:kg-one-2ord-tindep},
which implies that
\begin{equation}
\label{eq:part-sol-comp-1}
u(t, x) = w(t', x') = \eee^{\gamma\nu_k(t - \beta \cdot x)}\phi_k(\Lambda_\beta(x - \beta t - \xi))
\end{equation}
is a solution of \eqref{eq:kg-one-2ord}. Now we observe that
\begin{equation}
\label{eq:exp-manip}
\eee^{\gamma\nu_k(t - \beta \cdot x)} = \eee^{-\gamma\nu_k\beta\cdot\xi}\eee^{-\gamma|\beta|^2\nu_k t}\eee^{\gamma\nu_k t}
\eee^{-\gamma\nu_k\beta\cdot(x - \beta t - \xi)}
= \eee^{-\gamma\nu_k\beta\cdot\xi}\eee^{\frac{\nu_k}{\gamma}t}\eee^{-\gamma\nu_k\beta\cdot(x - \beta t - \xi)}.
\end{equation}
The first factor is constant and can be discarded. The second factor is the exponential growth factor in \eqref{eq:part-sol-2}. Finally,
$
\eee^{-\gamma\nu_k\beta\cdot(x - \beta t - \xi)}\phi_k(\Lambda_\beta(x - \beta t - \xi))
$
is precisely the first component of $\cY_{k, \beta}^+$. The second component of $\cY_{k, \beta}^+$ is found by computing
the time derivative of \eqref{eq:part-sol-comp-1}:
\begin{equation*}
\dd t u(t, x) = \eee^{\gamma\nu_k(t - \beta\cdot x)}\big(\gamma \nu_k\phi_k(\Lambda_\beta(x - \beta t - \xi)) - (\Lambda_\beta \beta)\cdot \grad\phi_k(\Lambda_\beta(x - \beta t - \xi))\big).
\end{equation*}
Using again \eqref{eq:exp-manip}, we see that the second component of $\eee^{\frac{\nu_k}{\gamma}t}\cY_{k, \beta}^+(x - \beta t - \xi)$ is indeed the time derivative of the fist component. One can treat \eqref{eq:part-sol-1} and \eqref{eq:part-sol-3} similarly.

If $\bs v(t)$ and $\bs u(t)$ are solutions of \eqref{eq:kg-one-ham}, then, using the fact that $H(t)$ is self-adjoint,
$J$ is skew-adjoint and $J^2 = -\Id$, we get
\begin{equation*}
\dd t\la J\bs v(t), \bs u(t)\ra = \la J^2H(t)\bs v(t), \bs u(t)\ra + \la J\bs v(t), JH(t) \bs u(t)\ra = 0.
\end{equation*}
Taking $\bs v(t) = \eee^{\frac{\nu_k}{\gamma}t}\cY_{k, \beta}^+(x - \beta t - \xi)$ we obtain \eqref{eq:covector-1}.
Similarly, \eqref{eq:covector-2} follows by considering $\bs v(t) = \eee^{-\frac{\nu_k}{\gamma}t}\cY_{k, \beta}^-(x - \beta t - \xi)$, whereas for \eqref{eq:covector-3} we take $\bs v(t) = \cY_{m, \beta}^0(x - \beta t - \xi)$.
\end{proof}

\subsection{Energy estimates}
\label{ssec:kg-one-energy}
Consider the following quadratic form, also appearing in \cite{CoMu2014, CoMa-multi}:
\begin{equation*}
\label{eq:Hbeta}
Q_\beta(\xi; \bs u_0, \bs u_0) := \frac 12\int_{\bR^d}\big((\dot u_0)^2 +2\dot u_0(\beta\cdot\grad u_0) + |\grad u_0|^2 + (1+V_\beta(\cdot - \xi))u_0^2\big)\ud x.
\end{equation*}
We have the following \emph{coercivity property}, proved in \cite{CoMa-multi}.
\begin{proposition}\label{prop:coer-kg-one}\cite[Proposition 3]{CoMa-multi}
For any $\beta \in (-1, 1)$ there exists $c > 0$ such that for all $\xi \in \bR^d$ and $\bs u_0 \in H^1 \times L^2$
the following bound is true:
\begin{equation}
\label{eq:kg-one-coer}
\begin{aligned}
Q_\beta (\xi; \bs u_0, \bs u_0) &\geq c\|\bs u_0\|_{H^1 \times L^2}^2 \\
&- \frac 1c \bigg(\sum_{k=1}^{K}
\la \alpha_{k, \beta}^-(\cdot - \xi), \bs u_0\ra^2 + \sum_{k=1}^{K}\la \alpha_{k, \beta}^+(\cdot - \xi), \bs u_0\ra^2
+ \sum_{m=1}^{M}\la \alpha_{m, \beta}^0(\cdot - \xi), \bs u_0\ra^2\bigg).
\end{aligned}
\end{equation}
\end{proposition}
\begin{proof}
For the convenience of the reader, we will provide a proof, different from the one given in \cite{CoMa-multi}.

Let $Y_{\beta, \xi} \subset H^1\times L^2$ be defined by
\begin{equation*}
Y_{\beta, \xi} := \{\bs u_0: \la \alpha_{k, \beta}^-(\cdot - \xi), \bs u_0\ra = \la \alpha_{k, \beta}^+(\cdot - \xi), \bs u_0\ra = \la \alpha_{m, \beta}^0(\cdot - \xi), \bs u_0\ra = 0\text{ for all }m, k\}.
\end{equation*}
Since $\bs u_0 \mapsto \int_{\bR^d}\big((\dot u_0)^2 +2\dot u_0(\beta\cdot\grad u_0) + |\grad u_0|^2 + u_0^2\big)\ud x$
defines a norm equivalent to the $H^1\times L^2$ norm, a standard weak convergence argument shows that it suffices to prove
\begin{equation*}
\label{eq:coer-kg-one-1}
\bs u_0 \in Y_{\beta, \xi} \Rightarrow Q_{\beta}(\xi; \bs u_0, \bs u_0) > 0.
\end{equation*}
We will construct a continuous one-to-one linear map $T_{\beta, \xi}: Y_{\beta, \xi} \to Y_{0, 0}$ such that
\begin{equation}
\label{eq:coer-kg-one-2}
Q_{\beta}(\xi; \bs u_0, \bs u_0) = \gamma^{-1}Q_{0}(0; T_{\beta, \xi}\bs u_0, T_{\beta, \xi}\bs u_0),\quad\text{for all }\bs u_0 \in Y_{\beta, \xi}.
\end{equation}
Now it easily follows from \eqref{eq:H1-coer} that $\bs w_0 \in Y_{0, 0} \Rightarrow Q_{0}(0; \bs w_0, \bs w_0) > 0$,
so this will finish the proof.

Let $\bs u_0 \in Y_{\beta, \xi}\cap (C_0^\infty\times C_0^\infty)$ and let $u(t, x)$ be the solution of \eqref{eq:kg-one-2ord} with the initial conditions
$(u(0, \cdot), \partial_t u(0, \cdot)) = \bs u_0$. Let $w(t, x)$ be defined by $w(t', x') = u(t, x)$,
where $t', x'$ are given by \eqref{eq:t'x'def}. We set $T_{\beta, \xi}\bs u_0 := (w(0, \cdot), \partial_t w(0, \cdot))$.

By the Chain Rule, we have
\begin{equation*}
\partial_t u = \gamma(\partial_{t'}w - \beta\cdot\grad_{x'}w), \qquad \grad_x u = -\gamma\beta\partial_{t'}w + \Lambda_\beta \grad_{x'}w,
\end{equation*}
and after a somewhat tedious computation we arrive at
\begin{equation}
\label{eq:tedious}
\begin{aligned}
&(\partial_t u)^2 +2\partial_t u(\beta\cdot\grad_x u) + |\grad_x u|^2 + (1+V_\beta(\cdot - \xi))u^2 \\
&= (\partial_{t'} w)^2 - 2\partial_{t'}w(\beta\cdot\grad_{x'}w) + |\grad_{x'}w|^2 + (1+V)w^2.
\end{aligned}
\end{equation}
Let $P$ be the hyperplane of the $(t', x')$ spacetime defined by $t' + \gamma\beta\cdot\xi = -\beta\cdot x'$ and let $\vd \sigma$
be the measure inherited from the Lebesgue measure.
In $(t, x)$ coordinates, $P$ is the hyperplane $t = 0$, so taking into account the change of measure and
\eqref{eq:tedious} we obtain
\begin{equation*}
Q_{\beta}(\xi; \bs u_0, \bs u_0) = \sqrt{\frac{1+\beta^2}{1-\beta^2}}\int_P \big((\partial_{t'} w)^2 - 2\partial_{t'}w(\beta\cdot\grad_{x'}w) + |\grad_{x'}w|^2 + (1+V)w^2\big)\ud \sigma.
\end{equation*}
We can now use the Divergence Theorem for the vector field $G(w)$ in the region of the $(t', x')$ spacetime
delimited by $t' = 0$ and $P$. This leads to \eqref{eq:coer-kg-one-2}.

Next, we need to prove that $T_{\beta, \xi}\bs u_0 \in Y_{0, 0}$.
For this purpose, we integrate the vector field $F(w_1, w_2)$ from Lemma~\ref{lem:kg-div-free}
with $w_1(t') := w(t')$ and $w_2(t') := \eee^{\nu_k t'}(\phi_k, \nu_k \phi_k)$,
in the region between $\{t' = 0\}$ and $P$. We have $\partial_{t'}w_2(0) = \nu_k\phi_k$, hence the boundary term corresponding to $t' = 0$ equals
\begin{equation*}
-\int_{\bR^d}\big(w_1(0)\partial_{t'}w_2(0) - w_2(0)\partial_{t'}w_1(0)\big)\ud x' = -\la (-\nu_k\phi_k, \phi_k), (w(0), \partial_t w(0))\ra
= -\la \alpha_{k, 0}^-, T_{\beta, \xi}\bs u_0\ra.
\end{equation*}
For $(t', x') \in P$ we have
\begin{equation*}
\begin{aligned}
&w_1\partial_{t'}w_2 - w_2\partial_{t'}w_1 - w_1\beta\cdot \grad w_2 + w_2\beta\cdot\grad w_1 \\
&= w(t', x')\nu_k\eee^{\nu_k t'}\phi_k(x') - \partial_t w(t', x')\eee^{\nu_k t'}\phi_k(x') \\
&- w(t', x')\eee^{\nu_k t'}\beta\cdot \grad\phi_k(x') + \beta \cdot \grad w(t', x')\eee^{\nu_k t'}\phi_k(x') \\
&= -\eee^{\nu_k t'}\phi_k(x')(\partial_{t'}w(t', x') - \beta\cdot\grad w(t', x')) + (\nu_k\eee^{\nu_k t'}\phi_k(x') - \eee^{\nu_k t'}\beta\cdot\grad \phi_k(x'))w(t', x') \\
&= -\frac{1}{\gamma}\alpha_{k, \beta}(x-\xi)\cdot \bs u_0(x).
\end{aligned}
\end{equation*}
If $\bs u_0 \in Y_{\beta, \xi}$, we deduce that the boundary term over $P$ equals $0$, thus the boundary term over $\{t'=0\}$
equals $0$ as well. Orthogonality to $\alpha_{k, 0}^+$ and $\alpha_{m, 0}^0$ are checked similarly, and we obtain $T_{\beta, \xi}\bs u_0 \in Y_{0, 0}$.

From \eqref{eq:coer-kg-one-2} and the coercivity of $Q_{0}(0; \bs w_0, \bs w_0)$ for $\bs w_0 \in Y_{0, 0}$ we deduce that
$T_{\beta, \xi}: Y_{\beta, \xi} \to Y_{0, 0}$ is continuous for the $H^1 \times L^2$ norm.
Thus, we can extend it by continuity from $C_0^\infty \times C_0^\infty$ to $T_{\beta, \xi}$.
In order to prove that it is one-to-one, we need to check that if
$T_{\beta, \xi}\bs u_n \to 0$ in $H^1 \times L^2$, then $\bs u_n \to 0$ in $H^1 \times L^2$.
Let $\bs w_n := T_{\beta, \xi}\bs u_n$ and let $w_n(t', x')$ be the corresponding solution of \eqref{eq:kg-one-2ord-tindep}.
We apply the Divergence Theorem to the vector field $((\partial_{t'} w_n)^2 + \|\grad_{x'} w_n\|^2 + w_n^2, -2\partial_{t'}w_n\grad_{x'}w_n)$,
in the region $\Omega$ contained between $\{t' = 0\}$ and $P$. The divergence equals $-V(x')\partial_{t'}w_n(t', x')$,
and we see that the exponential decay of $V$ implies $V \in L^1_{t'}L^2_{x'}(\Omega)$, thus
\begin{equation*}
\int_\Omega |V(t', x')||\partial_{t'}w_n(t', x')|\ud x'\ud t' \lesssim \|\partial_{t'}w_n\|_{L^\infty L^2} \to 0\quad\text{as }n \to \infty,
\end{equation*}
so we obtain
\begin{equation*}
\int_P \big((\partial_{t'} w)^2 - 2\partial_{t'}w(\beta\cdot\grad_{x'}w) + |\grad_{x'}w|^2 + w^2\big)\ud \sigma \to 0\quad\text{as }n \to \infty.
\end{equation*}
After a change of variables, this yields
\begin{equation*}
\int_{\bR^d}\big((\dot u_n)^2 +2 \dot u_n (\beta\cdot\grad u_n) + |\grad_x u_n|^2 + u_n^2\big)\ud x \to 0,
\end{equation*}
which finishes the proof.
\end{proof}
\begin{remark}
The quantity $Q_{0}(0; \bs w_0, \bs w_0)$ is the energy of \eqref{eq:kg-one-2ord-tindep}, and from the above considerations
it easily follows that $Q_{\beta}(\beta t + \xi; \bs u(t), \bs u(t))$ is constant for any solution $\bs u(t)$ of \eqref{eq:kg-one-ham}.
This can also be checked by a direct computation, which is the method we will have to adopt below in the case of multiple potentials.
\end{remark}
\subsection{Many potentials}
\label{ssec:kg-many}
We consider the linear Klein-Gordon equation with a finite number of moving potentials.

Let $V_j$ be a smooth exponentially decaying potential for $j \in \{1, 2, \ldots, J\}$,
such that $L_j := -\Delta + V_j$ has $K_j$ strictly negative eigenvalues $-\nu_{j, k}^2$ (for $k = 1, \ldots, K_j$)
and $\dim\ker L_j = M_j$.

Let $y_j(t)$ be positions of the potentials. We denote $\beta_j(t) := y_j'(t)$.
We write $\bs\beta(t) = (\beta_1(t), \ldots, \beta_J(t))$, $\bs y(t) = (y_1(t), \ldots, y_J(t))$.
We consider the equation
\begin{equation}
\label{eq:kg-many-2ord}
\partial_t^2 u = \Delta u - u - \sum_{j=1}^J (V_j)_{\beta_j(t)}(\cdot - y_j(t))u.
\end{equation}
Note that the Lorentz transformation is applied to the potentials $V_j$, according to their instantaneous velocity.

Fix $j \in \{1, 2, \ldots, J\}$ and let $\cY_{k, \beta}^-$, $\cY_{k, \beta}^+$, $\cY_{m, \beta}^0$, $\alpha_{k, \beta}^-$, $\alpha_{k, \beta}^+$, $\alpha_{m, \beta}^0$ be the functions defined in Paragraph~\ref{ssec:kg-one} for $V_j$ instead of $V$.
We denote
\begin{align*}
\cY_{j, k}^-(t) &:= \cY_{k, \beta(t)}^-(\cdot - y_j(t)), \\
\cY_{j, k}^+(t) &:= \cY_{k, \beta(t)}^+(\cdot - y_j(t)), \\
\cY_{j, m}^0(t) &:= \cY_{m, \beta(t)}^0(\cdot - y_j(t)), \\
\alpha_{j, k}^-(t) &:= \alpha_{k, \beta(t)}^-(\cdot - y_j(t)), \\
\alpha_{j, k}^+(t) &:= \alpha_{k, \beta(t)}^+(\cdot - y_j(t)), \\
\alpha_{j, m}^0(t) &:= \alpha_{m, \beta(t)}^0(\cdot - y_j(t)), \\
V(t) &:= \sum_{j=1}^J (V_j)_{\beta_j(t)}(\cdot - y_j(t)),
\end{align*}
where $k \in \{1, \ldots, K_j\}$ and $m \in \{1, \ldots, M_j\}$.


If we let
\begin{equation*}
H(t) := \begin{pmatrix} -\Delta +1 + \sum_{j=1}^J (V_j)_{\beta_j(t)}(\cdot - y_j(t)) & 0 \\ 0 & 1 \end{pmatrix},
\end{equation*}
then \eqref{eq:kg-many-2ord} can be written as
\begin{equation}
\label{eq:kg-many-ham}
\partial_t \bs u(t) = JH(t) \bs u(t).
\end{equation}
By standard arguments based on energy estimates, this equation defines a strongly continuous evolution operator
in $H^1 \times L^2$, which we denote $T(\tau, t)$.

In order to define the relevant quadratic form $Q$, we need to use cut-offs, cf. \cite[Section 3.5]{CoMa-multi}.
We let $\chi: \bR^d \to \bR$ be a $C^\infty$ function such that
\begin{equation*}
\label{eq:chi}
\chi(x) = 0\ \text{for}\ |x| \geq \frac 12,\quad \chi(x) = 1\ \text{for}\ |x| \leq \frac 14,\quad 0 \leq \chi(x) \leq 1\ \text{for}\ x\in \bR^d.
\end{equation*}
Assume $|y_l(t) - y_j(t)| \geq \frac{1}{\eta}$ for $j \neq l$ and $t \geq t_0$ and some (small) $\eta > 0$.
We set
\begin{equation*}
\chi_j(t, x) := \chi\big(\eta(x - y_j(t))\big)
\end{equation*}
and we define
\begin{equation*}
\label{eq:Hbeta-many}
Q(t; \bs u_0, \bs u_0) := \frac 12\int_{\bR^d}\Big((\dot u_0)^2 +2\sum_{j=1}^J\chi_j(t)\dot u_0(\beta_j(t)\cdot\grad u_0) + |\grad u_0|^2 + (1+V(t))u_0^2\Big)\ud x.
\end{equation*}
Note that similar localised functionals were used by Martel, Merle and Tsai in \cite{MMT02, MMT06}.

\begin{lemma}
There exists $c > 0$ such that for all $\bs u_0 \in H^1 \times L^2$ the following bound is true:
\begin{equation*}
\label{eq:kg-many-coer}
Q(t; \bs u_0, \bs u_0) \geq c\|\bs u_0\|_{H^1 \times L^2}^2 -\frac 1c \sum_{j=1}^J\bigg(\sum_{k=1}^{K_j}
\la \alpha_{j,k}^-(t), \bs u_0\ra^2 + \sum_{k=1}^{K_j}\la \alpha_{j,k}^+(t), \bs u_0\ra^2
+ \sum_{m=1}^{M_j}\la \alpha_{j,m}^0(t), \bs u_0\ra^2\bigg).
\end{equation*}
\end{lemma}
\begin{proof}
For $j \in \{1, \ldots, J\}$, let $\bs u_j := \chi_j \bs u_0$.
We obtain the result by summing \eqref{eq:kg-one-coer}, applied for $\bs u_j$ instead of $\bs u_0$ for $j \in \{1, \ldots, J\}$.
\end{proof}
\begin{proposition}
\label{prop:kg-main}
Let $v < 1$, $\nu := \min \{ \nu_{j, k} \}$ and $K := \sum_{j=1}^J K_j$.
For any $\epsilon > 0$ there exists $\eta > 0$ such that if for $t$ large enough
\begin{equation}
\label{eq:kg-param-hyp}
|\beta_j(t)| \leq v, \qquad |\beta_j'(t)| \leq \eta,\qquad |y_j(t) - y_l(t)| \geq \frac{1}{\eta}\qquad\text{for all }j \neq l,
\end{equation}
then the semigroup $T(\tau, t)$ has an exponential dichotomy with exponents $\epsilon$ and $\nu\sqrt{1-v^2} - \epsilon$.
Moreover, $\codim X_\tx s = K$.
\end{proposition}
Before giving a proof, we need one more lemma about a dynamical control of stable and unstable directions.
Let $\bs u(t)$ we a solution of \eqref{eq:kg-many-ham}.
The stable and unstable components are defined by
\begin{equation*}
a_{j, k}^\pm(t) := \la \alpha_{j, k}^\pm(t), \bs u(t)\ra, \quad j \in \{1, \ldots, J\},\ k \in \{1, \ldots, K_j\}.
\end{equation*}
\begin{lemma}
\label{lem:kg-unstable}
For any $c > 0$ there exists $\eta > 0$ such that if \eqref{eq:kg-param-hyp} holds, then for all $t$
\begin{align*}
\Big|\dd t a_{j, k}^\pm(t) \mp \frac{\nu_{j,k}}{\gamma_j}a_{j,k}^\pm(t)\Big| &\leq c\|\bs u(t)\|_{H^1\times L^2}, \qquad\text{for all }j\text{ and }1 \leq k \leq K_j, \\
\Big|\dd t a_{j, m}^0(t)\Big| &\leq c\|\bs u(t)\|_{H^1\times L^2}, \qquad\text{for all }j\text{ and }1 \leq m \leq M_j.
\end{align*}
\end{lemma}
\begin{proof}
We prove the first bound for the sign ``$-$'', the remaining cases being similar. Fix $t_0$ and let $\beta := \beta_j(t_0)$,
$\xi := y_j(t_0)$. Let $\alpha_{k, \beta}^-$ be defined by \eqref{eq:kg-am-def}
and let $H_\beta$ be defined by \eqref{eq:Hbeta-def} with $V = V_j$. Then \eqref{eq:beta-derivative} yields
\begin{equation*}
\Big\la \dd t\alpha_{j, k}^-(t_0), \bs u(t_0)\Big\ra = \Big\la \dd t\Big\vert_{t=0}\alpha_{k, \beta}^-(\cdot - \beta t - \xi), \bs u(t_0)\Big\ra + O(\eta \|\bs u(t_0)\|_{L^2}).
\end{equation*}
We also have
\begin{equation*}
\la \alpha_{j, k}^-(t_0), \partial_t \bs u(t_0)\ra = \la \alpha_{k, \beta}^-(\cdot - \xi), JH_\beta(\xi) \bs u(t_0)\ra
+ \sum_{j' \neq j}\big\la \alpha_{k, \beta}^-(\cdot - \xi), (V_{j'})_{\beta_{j'}(t_0)}(\cdot - y_{j'}(t_0))\bs u(t_0)\big\ra.
\end{equation*}
If \eqref{eq:kg-param-hyp} holds with $\eta \ll 1$, then the second term above is $\ll \|\bs u(t_0)\|_{L^2}$ when $t_0 \gg 1$
(similarly as in the proof of Proposition~\ref{prop:heat}).
We thus obtain
\begin{equation*}
\dd t a_{j, k}^-(t_0) = \Big\la \dd t\Big\vert_{t=0}\alpha_{k, \beta}^-(\cdot - \beta t - \xi), \bs u(t_0)\Big\ra + \la \alpha_{k, \beta}^-(\cdot - \xi), JH_\beta(\xi) \bs u(t_0)\ra + o(\|\bs u(t_0)\|_{L^2}),
\end{equation*}
and the conclusion follows from \eqref{eq:covector-1}.
\end{proof}
\begin{proof}[Proof of Proposition~\ref{prop:kg-main}]
We set
\begin{align*}
I_t^+(\bs u(t)) &:= \bigg(\sum_{j=1}^J \sum_{k=1}^{K_j} |a_{j, k}^+(t)|^2\bigg)^{\frac 12}, \\
I_t^-(\bs u(t)) &:= \bigg(\max\Big(0, Q(t; \bs u(t), \bs u(t)) + \sum_{j=1}^J \sum_{k=1}^{K_j} |a_{j, k}^-(t)|^2 +  \sum_{j=1}^J \sum_{m=1}^{M_j} |a_{j, m}^0(t)|^2\Big)\bigg)^\frac 12.
\end{align*}
We need to verify the assumptions of Theorem~\ref{thm:main-cont-2}, with $\epsilon$ instead of $\lambda$ and $\nu - \epsilon$ instead of $\mu$.
As in the case of the heat equation, this boils down to showing that
\begin{align}
\dd t I_t^+(\bs u(t)) &\geq \nu\sqrt{1-v^2} I_t^+(\bs u(t)) - \wt\epsilon\|\bs u(t)\|_{H^1\times L^2}\qquad\text{if }I_t^+(\bs u(t)) \geq c_4 I_t^-(\bs u(t)),\label{eq:kg-dtIp} \\
\dd t I_t^-(\bs u(t)) &\leq \wt\epsilon\|\bs u(t)\|_{H^1 \times L^2}\qquad \text{if }I_t^+(\bs u(t)) \leq c_3 I_t^-(\bs u(t)), \label{eq:kg-dtIm}
\end{align}
where $\wt \epsilon \to 0$ when $\eta \to 0$.
Inequality \eqref{eq:kg-dtIp} follows from Lemma~\ref{lem:kg-unstable}.
When proving \eqref{eq:kg-dtIm}, we can assume that $I_t^-(\bs u(t)) > 0$, because otherwise $\bs u(t) = 0$.
From Lemma~\ref{lem:kg-unstable} we obtain
\begin{equation*}
\dd t\Big(\sum_{j=1}^J \sum_{k=1}^{K_j} |a_{j, k}^-(t)|^2 +  \sum_{j=1}^J \sum_{m=1}^{M_j} |a_{j, m}^0(t)|^2\Big) \leq \wt\epsilon\|\bs u(t)\|^2, \qquad \wt\epsilon \ll 1\text{ when }\eta \to 0,
\end{equation*}
so we are left with computing $\dd t Q(t; \bs u(t), \bs u(t))$. By density, we can assume the solution is smooth.

Fix $t_0$ and let $\beta_j := \beta_j(t_0)$, $\xi_j := y_j(t_0)$, $\bs u_j := \chi_j \bs u(t_0)$.
%
%
We have
\begin{equation}
\label{eq:partial-t-Q-2}
\begin{aligned}
\partial_t Q(t_0; \bs u(t_0), \bs u(t_0)) \simeq \frac 12\partial_t V(t_0)u(t_0)^2\ud x
&\simeq \frac 12\sum_{j=1}^J \partial_{t=t_0}\big((V_j)_{\beta_j(t)}(\cdot - y_j(t))\big)u(t_0)^2 \\
&\simeq \frac 12\sum_{j=1}^J\beta_j\cdot\grad\big((V_j)_{\beta_j}\big)(\cdot - \xi_j)u(t_0)^2,
\end{aligned}
\end{equation}
where the passage from the first to the second line is justified by the rapid decay of the potentials.

Next, we compute
\begin{equation*}
\begin{aligned}
2Q(t_0; \bs u(t_0), JH(t_0)\bs u(t_0)) &= \int_{\bR^d}\bigg( \dot u(t_0)\Big(\Delta u(t_0) - u(t_0) - \sum_{j=1}^J (V_j)_{\beta_j}(\cdot - \xi_j)u(t_0)\Big) \\
&+ \sum_{j=1}^J\chi_j(t_0)\Big(\Delta u(t_0) - u(t_0) - \sum_{l=1}^J (V_l)_{\beta_l}(\cdot - \xi_l)u(t_0)\Big)(\beta_j\cdot\grad u(t_0)) \\
&+ \sum_{j=1}^J\chi_j(t_0)\dot u(t_0)(\beta_j\cdot\grad \dot u(t_0)) + \grad u_0\cdot\grad\dot u(t_0) \\
&+ \Big(1+\sum_{j=1}^J (V_j)_{\beta_j}(\cdot - \xi_j)\Big)u(t_0)\dot u(t_0)\bigg)\ud x.
\end{aligned}
\end{equation*}
We integrate by parts and note that whenever the differentiation falls on the cut-off function, we obtain a negligible term.
We obtain
\begin{equation*}
2Q(t_0; \bs u(t_0), JH(t_0)\bs u(t_0)) \simeq -\sum_{j=1}^J\sum_{l=1}^J\chi_j(t_0)(V_l)_{\beta_l}(\cdot - \xi_l)u(t_0)\beta_j\cdot\grad u(t_0).
\end{equation*}
In this sum, the terms for which $l \neq j$ are negligible because of the fast decay of the potentials.
For the same reason, for $l = j$ we can neglect the cut-off function. We thus have
\begin{equation*}
2Q(t_0; \bs u(t_0), JH(t_0)\bs u(t_0)) \simeq -\sum_{j=1}^J(V_j)_{\beta_j}(\cdot - \xi_j)u(t_0)\beta_j\cdot\grad u(t_0).
\end{equation*}
Comparing with \eqref{eq:partial-t-Q-2}, we obtain
\begin{equation*}
\Big|\dd t\big\vert_{t = t_0}Q(t; \bs u(t), \bs u(t))\Big| \ll \|\bs u(t_0)\|_{H^1 \times L^2}^2.
\end{equation*}
\end{proof}

\bibliographystyle{plain}
\bibliography{researchbib}

\end{document}